\documentclass[11pt]{article}

\usepackage[english]{babel}
\usepackage{graphicx}
\usepackage[colorlinks]{hyperref}
\usepackage[dvipsnames]{xcolor}

\usepackage{geometry}
\usepackage[backgroundcolor=white,linecolor=red,bordercolor=red]{todonotes}
\usepackage[shortlabels]{enumitem}
\setenumerate{label=(\roman*)}

\usepackage{hyperref}
\newcommand\myshade{100}
\hypersetup{
  linkcolor  = RoyalBlue!\myshade!black,
  citecolor  = ForestGreen!\myshade!black,
  urlcolor   = RedOrange!\myshade!black,
  colorlinks = true,
}

\usepackage{amsmath}
\usepackage{amssymb}
\usepackage{amsfonts}
\usepackage{amsthm}
\usepackage{mathtools}
\usepackage{bm}
\usepackage{bbold}
\usepackage{tikz-cd}

\usepackage[capitalise]{cleveref}
\crefformat{equation}{(#2#1#3)}
\crefname{subsection}{Section}{Sections}

\theoremstyle{plain}
\newtheorem{thm}{Theorem}[section]
\newtheorem{lem}[thm]{Lemma}
\newtheorem{prop}[thm]{Proposition}
\newtheorem{cor}[thm]{Corollary}

\theoremstyle{definition}
\newtheorem{defn}[thm]{Definition}
\newtheorem{example}[thm]{Example}
\newtheorem{question}[thm]{Question}

\theoremstyle{remark}
\newtheorem*{rem}{Remark}

\makeatletter
\def\@fnsymbol#1{\ensuremath{\ifcase#1\or \dagger\or \ddagger\or
           \dagger\dagger
           \or \ddagger\ddagger \else\@ctrerr\fi}}
\makeatother

\newcommand\N{\ensuremath{\mathbb{N}}}
\newcommand\R{\ensuremath{\mathbb{R}}}

\renewcommand\O{\ensuremath{\varnothing}}

\newcommand\C{\ensuremath{\mathbb{C}}}

\newcommand{\el}[1]{\ensuremath{\ell_{#1}}}

\newcommand{\Lpos}[1]{\ensuremath{\mathcal{L}_+(#1)}}
\newcommand{\Lr}[1]{\ensuremath{\mathcal{L}_r(#1)}}

\DeclareMathOperator{\barespn}{span}

\DeclareMathOperator{\range}{range}
\DeclareMathOperator{\Fin}{Fin}

\usepackage[giveninits=true,maxbibnames=99]{biblatex}
\usepackage{csquotes}

\newcommand{\keywords}[2]{%
        \textbf{Keywords} #1 \\
        \textbf{MSC2020} #2
}

\DeclareMathOperator*{\olim}{o-lim}

\title{Band projections and order idempotents\\in Banach lattice
algebras}
\author{David Muñoz-Lahoz\footnotemark[2]}
\date{}

\begin{document}

\maketitle

\begin{abstract}
    Motivated by recent work about band projections on spaces of
    regular operators over a Banach lattice, given a Banach lattice
    algebra $A$, we will say an element $a \in A_+$ is a band
    projection if the multiplication operator $L_aR_a\in \mathcal
    L_r(A)$ is a band projection. Our aim in this note is to explore
    the relations between this and the notion of order idempotent
    (those elements $a$ in a Banach lattice algebra $A$ with identity
    $e$ such that $0\leq a\leq e$ and $a^2=a$). We also revisit
    the properties of the ideal generated by the identity on a Banach
    lattice algebra, motivated by those of the centre of a Banach
    lattice.
\end{abstract}
{46B42; 06F25; 46A32; 47A10; 47B65; 47L10}

\footnotetext[2]{\texttt{david.munnozl (at) uam (dot) es}

Instituto de Ciencias Matemáticas (CSIC-UAM-UC3M-UCM) Consejo Superior
de Investigaciones Científicas, C/ Nicolás Cabrera, 13–15, Campus de
Cantoblanco UAM 28049, Madrid, Spain.}


\section{Introduction}

This is a follow-up of the paper \cite{munoz-lahoz_tradacete2024},
where we studied band projections in spaces of regular operators. Our
aim here is to explore analogous results in the more general context
of Banach lattice algebras. Besides contributing to the theory of
these structures, this serves the purpose of showing that some
properties of spaces of regular operators hold for arbitrary Banach
lattice algebras.

Given a Banach lattice $X$, we can order the space of
operators on $X$ using the cone of positive operators,
that is, setting an operator $T\colon X\to X$ to be positive if
$T(X_+)\subseteq X_+$. In general, however, not every operator can be
expressed as the difference of two positive operators, so the positive
cone need not be generating. Those operators that are
in the span of the positive cone are called regular operators.

The space of regular operators $\Lr X$ is an ordered vector space, but
it need not be a lattice. It is a remarkable result by Freudenthal,
Kantorovich and Riesz that when $X$ is order complete, $\Lr X$
is again a Banach lattice when equipped with the regular norm (defined
as $\|T\|_r=\|\,|T|\,\|$ for $T \in \Lr X$). Since the composition of positive operators
is again positive, $\Lr X$ is also a Banach algebra.

So in $\Lr X$ a lattice and an algebraic structure coexist, with the
compatibility condition that the product of positive elements is
positive. In \cite{munoz-lahoz_tradacete2024} we used this structure
to study band projections on $\Lr X$. More precisely,
if $\{P_\lambda \}_{\lambda \in \Lambda }$ is a family of pairwise disjoint band
projections on $X$, and $\Gamma \subseteq \Lambda \times \Lambda $, we
defined for positive operators $T \in \Lpos X$ the map
\begin{equation}\label{eq:inner_intro}
\mathcal{P}_\Gamma (T)=\bigvee_{(\alpha ,\beta )\in \Gamma } P_\alpha
TP_\beta.
\end{equation}
We showed that $\mathcal{P}_\Gamma \colon \Lpos X\to \Lpos X$
extends to a unique band projection on $\Lr X$, and called band
projections of this form inner.

Even though the definition of inner band projections uses only the
lattice and algebraic operations of $\Lr X$, the proof that
\cref{eq:inner_intro} actually defines a band projection relied
heavily on the expression of the Riesz--Kantorovich formulae for the
lattice operations in $\Lr X$. One of
the main motivations for this paper is to determine whether we can
extend inner band projections to more general spaces in which a
lattice and an algebraic operation coexist in a manner analogous to
that of $\Lr X$.

This leads to the notion of Banach lattice algebras. Formally, a Banach lattice algebra is a Banach lattice together with a
Banach algebra structure in which the product of positive elements is
positive. One of the most important open problems in the theory of
Banach lattice algebras is the following:
\begin{question}[{\cite[Question 5.3]{wickstead2017_questions}}]
    Is every Banach lattice algebra isometrically (isomorphic) to a
    closed subalgebra and sublattice of some space of regular
    operators?
\end{question}
It is relevant to this question, and therefore to the theory, to
understand which properties of spaces of regular operators depend only
on the Banach lattice algebra structure, and not on the fact that we
are actually dealing with operators. It is in this line of
investigation where our paper fits in.

The first step towards an abstract formulation of inner band
projections is to determine which elements of a Banach lattice algebra
$A$ should play the role of the band projections $\{P_\lambda
\}_{\lambda \in \Lambda }$. It is a classical result
that an operator $P\colon X\to X$ on a Banach lattice $X$ is a band
projection if and only if $P^2=P$ and $0\le P\le I_X$. When $A$
has an identity $e$, it seems then
natural to regard the elements $p \in
A$ satisfying $p^2=p$ and $0\le p\le e$ as the analogue of band projections. These
elements, called order idempotents, were already considered by E.\
A.\ Alekhno in the context of ordered Banach algebras (see
\cite{alekhno2012}). For
instance, Alekhno showed that the set of order idempotents forms a Boolean
algebra, just like the set of band projections on a Banach lattice
does. We
shall give an alternative proof of this fact, and show that
the set of order idempotents is order closed.

But for this we will
need to study the ideal generated by the identity element in $A$. This
ideal can be regarded as
the abstract analogue of the center of an order complete Banach
lattice.  We will collect results by L.\ Martignon
\cite{martignon1980}, C.\ B.\ Huijsmans \cite{huijsmans1995} and
E.\ Scheffold \cite{scheffold1984} to show that this ideal posseses
the most fundamental
properties of the center: it is lattice and algebra isometric to a
space of continuous functions on a compact Hausdorff space, it is a
projection band, and it is an inverse closed subalgebra of $A$.

And what happens if we don't have an identity? Can we still define
elements that play the role of band projections? Recently, it was
proved in \cite{munoz-lahoz_tradacete2024} that band projections are characterized by their
multiplication operators in the following way: a positive operator
$P\colon X\to X$ defined on an order complete Banach lattice $X$ is a
band projection if and only if the multiplication operator
\[
\begin{array}{cccc}
L_PR_P\colon& \Lr X & \longrightarrow & \Lr X \\
        & T & \longmapsto & PTP \\
\end{array}
\]
is a band projection. Band projections could therefore be defined as those
elements $p \in A_+$ for which the operator
\[
\begin{array}{cccc}
L_pR_p\colon& A & \longrightarrow & A \\
        & a & \longmapsto & pap \\
\end{array}
\]
is a band projection. This is analogue to the approach taken
by K.\ Vala in \cite{vala1967} to define compact elements in normed
algebras.

Several properties of band projections, in this sense, are
studied. It is shown that the set of band projections
need not be order closed or even commute; in particular, when we have
an identity, band projections and order idempotents may not be the
same. The relation between order idempotents and band projections
appears to be subtle, as it is illustrated with several
finite-dimensional examples. Some conditions guaranteeing that a band
projection is order idempotent are also given.

Much better behaved is the set of left and right band projections,
defined as those positive elements in $A$ for which both the left and right
multiplication operators are band projections.
These elements coincide with order idempotents when
an identity is present, and they always commute. Moreover, they
provide the appropiate generalization of band projections that allows
for the abstract construction of inner band projections we were looking
for. Our main result (\cref{thm:innerbp}) is that \eqref{eq:inner_intro} defines a band
projection in any order complete Banach
lattice algebra as long as the $\{P_\lambda \}_{\lambda \in
\Lambda }$
are taken to be left and right band projections satisfying $P_\alpha P_\beta
=\delta _{\alpha \beta }P_\alpha $ (where $\delta _{\alpha \beta }$ is
the Kronecker delta). To prove it, we cannot resort directly to the
Riesz--Kantorovich formulae, as done in
\cite{munoz-lahoz_tradacete2024}; instead, we have to work with the
multiplication operators $L_{P_\alpha }R_{P_\beta }$ to translate
properties from elements in $A$ to operators acting on $A$ (i.e., to elements
in $\Lr A$). Even before reaching inner band projections, many of our
results will use this technique, that allows for the generalization of
properties from spaces of regular operators to Banach lattice
algebras. We are not aware of any other systematic use of this idea in
the context of Banach lattice algebras, but one can find it already mentioned or
implicitly used in some parts of the literature (for instance, in
\cite{bernau_huijsmans1990_unit,huijsmans1988,martignon1980,wickstead2017_questions}).

The content of this note is organized in two big sections (besides this
introduction): \cref{sec:regop}, which collects some well-known facts
about spaces of regular operators, and \cref{sec:bla}, which
generalizes the notions in \cref{sec:regop} to Banach lattice
algebras. The reader familiar with regular operators can
jump directly to \cref{sec:bla} on page~\pageref{sec:bla}; the reader familiar with
regular operators and Banach lattice algebras can skim through
\cref{sec:bla_intro,sec:bla_ideal} to see the notation, and then start
reading at \cref{sec:bpoi} on page~\pageref{sec:bpoi}, where most
of the original results are contained.

These two big sections are, in turn, divided into subsections that
follow a parallel structure. A notion introduced for the regular
operators in a subsection of \cref{sec:regop} is then studied for
general Banach lattice algebras in the corresponding subsection of
\cref{sec:bla}. More precisely, we have the following bijection:
\begin{center}
    \begin{tabular}{ll}
        \textbf{2\ \ Spaces of regular operators} & \textbf{3\ \ Banach
        lattice algebras}\\
            2.1\ \ Regular operators & 3.1\ \ Banach lattice
            algebras\\
            2.2\ \ Band projections & 3.3\ \ Band projections and
            order idempotents\\
            2.3\ \ The center of a Banach lattice & 3.2\ \ The ideal
            generated by the unit element\\
            2.4\ \ Complexification and spectra & 3.4\ \
            Complexification and spectra\\
            2.5\ \ Inner band projections & 3.5\ \ Inner band
            projections\\
    \end{tabular}
\end{center}

\section{Spaces of regular operators}\label{sec:regop}

In this section we review some properties of the space of
regular operators over a Banach lattice that will motivate
our study of Banach lattice algebras. The reader familiar with the
topic can therefore skip it, except maybe for
\cref{thm:bp_multop,cor:bp_multop,sec:regop_innerbp}. With these
exceptions, all the contents are classical, and may be found in
\cite{abramovich_aliprantis2002,aliprantis_burkinshaw2006,meyer-nieberg1991}.

\subsection{Regular operators}\label{sec:regop_intro}

Let $X$ and $Y$ be Banach lattices, and let $\mathcal{L}(X,Y)$ denote
the set of (linear and bounded) operators from $X$ to $Y$. We say that an operator $T\colon
X\to Y$ is \emph{positive} if $T(X_+)\subseteq Y_+$. The set of
positive operators $\Lpos{X,Y}$ forms a cone in $\mathcal{L}(X,Y)$. With
this cone, $\mathcal{L}(X,Y)$ is an ordered Banach space, but it need
not be a lattice; in fact, it may even be the case that $\Lpos{X,Y}-\Lpos{X,Y}$ is a
proper subset of $\mathcal{L}(X,Y)$. The elements of
$\Lpos{X,Y}-\Lpos{X,Y}$ are called \emph{regular operators}, and the set
of regular operators is denoted by $\Lr{X,Y}$. Regular operators are
those operators that can be written as the difference of two positive
operators; equivalently, an operator $T$ is regular if there exists a
positive operator $S \in \Lpos{X,Y}$ such that $T\le S$.

The space of regular operators $\Lr{X,Y}$ need not be a Banach lattice either.
It was first shown by Freudenthal,
Kantorovich and Riesz that, when $Y$ is \emph{order complete} (that is, when any bounded above
set has supremum), $\Lr{X,Y}$ is a vector lattice. Later, an
appropiate norm was introduced in this
space, so as to make it a Banach lattice. More precisely:

\begin{thm}
    Let $X$ and $Y$ be Banach lattices. If $Y$ is order complete, then
    $\Lr{X,Y}$ is a Banach lattice, with supremum and infimum
    given by
    \begin{align}
        (S\vee T)x&=\sup \{\, Su+Tv : u,v \in X_+,\, u+v=x \,
        \},\label{eq:rk1}\\
        (S\wedge  T)x&=\inf \{\, Su+Tv : u,v \in X_+,\, u+v=x \,
        \},\label{eq:rk2}
    \end{align}
    where $S,T \in \Lr{X,Y}$ and $x \in X_+$, and with norm
    \begin{equation}\label{eq:regnorm}
        \|T\|_r=\|\,|T|\,\|.
    \end{equation}
    Moreover, $\Lr{X,Y}$ is order complete.
\end{thm}

Equations \cref{eq:rk1} and \cref{eq:rk2} and their derivatives (e.g., $|T|=\sup \{\,
|Tu| : -x\le u\le x \, \} $) are called the \emph{Riesz--Kantorovich
formulae}. The norm defined in \cref{eq:regnorm} is called the
\emph{regular norm}.

When $X=Y$, we will simplify our notation from $\mathcal{L}(X,X)$ to
$\mathcal{L}(X)$, from $\Lr{X,X}$ to $\Lr X$, and so on. In this case, the space $\mathcal{L}(X)$
is a Banach algebra with multiplication given by composition, and norm
given by operator norm. It is clear that the composition of positive operators is again
positive. From this, and the fact that the modulus is
submultiplicative (i.e., $|TS|\le |T| |S|$), it follows
that $\Lr X$ is a Banach algebra with respect to the regular norm. If $X$ is order
complete, we have seen that $\Lr X$ is a Banach lattice. The
simultaneous existence of these two structures will be our object of
study in \cref{sec:bla}.

\subsection{Band projections}

The interaction between the lattice and algebraic structures of $\Lr
X$ manifests itself clearly when we study band projections. A subspace $B$ of $X$ that is closed
under lattice operations (i.e., if $y\in B$ then $|y| \in B$)
is called a \emph{sublattice}. A sublattice
that is solid (i.e., if $|x|\le |y|$ and $y \in B$ then $x \in B$) is
called an \emph{ideal}. An ideal that is order closed (i.e., if
$(y_\gamma )_{\gamma}\subseteq B$ is an increasing net and
$\sup_\gamma y_\gamma $ exists in $X$, then $\sup_\gamma y_\gamma \in
B$) is called a \emph{band}. A band for which $X=B\oplus B^{d}$, where
\[
B^{d}=\{\, x \in X : |x|\wedge |y|=0\text{ for all }y\in B \, \},
\]
is called a \emph{projection band}. In this case we have a positive
projection $P\colon X\to X$ onto $B$, called a \emph{band projection}.
It turns out that band
projections can be nicely characterized in terms of the algebraic and
order structure in $\mathcal{L}(X)$.

\begin{thm}\label{thm:regop_bpchar}
    Let $X$ be a Banach lattice. An operator $P\colon X\to X$ is a band projection if and only if $P^2=P$ and
    $0\le P\le I_X$.
\end{thm}

In general, not all bands are projection bands. When $X$ is order complete,
however, every band is a projection band. In this case, we will talk
indistinguishably about bands, projection bands, and band projections.

Finite infima and suprema of band projections yield new band projections.
Surprisingly enough, lattice operations between band projections
are determined by the algebraic and linear operations only.

\begin{thm}\label{thm:regop_boole}
    Let $X$ be a Banach lattice. The set of band projections on $X$
    forms a Boolean algebra with operations
    \[
    \overline{P}=I_X-P,\quad P_1\vee P_2=P_1+P_2-P_1P_2,\quad P_1\wedge
    P_2=P_1P_2,
    \]
    minimum the zero map, and maximum the identity map.
    Moreover, suprema and infima coincide with those in $\Lr X$.
\end{thm}

A direct consequence of previous theorem is that band projections
commute. Note here the interplay between the lattice and algebraic
operations: we are able to express the lattice operations in purely
algebraic terms, and from this we are able to deduce an algebraic
property.

It will be central to our discussion that band projections are
characterized among positive operators by the multiplication
operators they define. Let $A,B \in \Lr X$ be regular operators on an
order complete Banach lattice $X$. The operators
\[
\begin{array}{cccc}
L_A\colon& \Lr X & \longrightarrow & \Lr X \\
        & T & \longmapsto & AT \\
\end{array}\quad
\begin{array}{cccc}
R_B\colon& \Lr X & \longrightarrow & \Lr X \\
        & T & \longmapsto & TB \\
\end{array}
\]
are called \emph{left} and \emph{right multiplication operators},
respectively. More generally, any operator of the form $L_AR_B$, for
some $A,B \in \Lr X$, is
called a \emph{multiplication operator}. As promised, using multiplication
operators we can characterize those elements in $\Lpos X$ that are band
projections as follows.

\begin{thm}[{\cite[Theorem 4.5]{munoz-lahoz_tradacete2024}}]\label{thm:bp_multop}
    Let $X$ be an order complete Banach lattice, and let $A,B \in \Lr
    X$ be nonzero regular operators. The multiplication operator
    $L_AR_B$ is a band projection if and only if there exists $\lambda
    \in \R\setminus\{0\}$ such that $\lambda A$ and $B/\lambda $ are
    band projections.
\end{thm}

\begin{cor}\label{cor:bp_multop}
    Let $X$ be an order complete Banach lattice, and let $A \in \Lpos
    X$. The multiplication operator $L_AR_A$ is a band projection
    if and only if $A$ is a band projection.
\end{cor}
\begin{proof}
    Suppose that $L_AR_A$ is a band projection. If $A=0$ the result
    follows,
    so assume $A\neq 0$. By previous theorem, there exists
    $\lambda \in \R\setminus\{0\}$ such that $\lambda A$ and
    $A/\lambda $ are band projections. Since $A$ is positive, $\lambda
    $ must also be positive. Using that $\lambda A$ is a
    projection,
    \[
        \left( \frac{\lambda A}{\lambda ^2} \right) ^2=\frac{(\lambda
        A)^2}{\lambda ^{4}}=\frac{\lambda A}{\lambda ^{4}}.
    \]
    But $\lambda A/\lambda ^2=A/\lambda $ is also a projection, so
    $(\lambda A/\lambda ^2)^2=A/\lambda$. Join both equations to
    arrive at $\lambda ^2A=A$. Since $A\neq 0$, it
    must be the case that $\lambda ^2-1=0$. But $\lambda >0$ implies $\lambda =1$, and therefore $A$ is a band projection.
\end{proof}

\subsection{The center of a Banach lattice}\label{sec:center}

Central operators and the center of a Banach lattice have been
extensively studied in the literature (see \cite[Section
3.3]{abramovich_aliprantis2002}). Here we only recall some basic
definitions and a couple of results that will be useful later.

\begin{defn}
    Let $X$ be a Banach lattice. An operator $T\colon X\to X$ is
    called \emph{central} if there exists some scalar $\lambda >0$
    such that $|Tx|\le \lambda |x|$ for all $x \in X$. The collection
    of all central operators is denoted $\mathcal{Z}(X)$ and is called
    the \emph{center of $X$}.
\end{defn}

For example, central operators on $C(K)$ are precisely the operators
of the form $M_f(g)=fg$, where $f,g \in C(K)$.  Note also that band
projections are positive central operators.

Central operators have nice properties: they are regular and order
continuous, and positive central operators are also lattice
homomorphisms. The following theorem collects some of the most
important properties of the center.

\begin{thm}[Wickstead \cite{wickstead1977}]\label{thm:center}
    Let $X$ be a Banach lattice.
    \begin{enumerate}
        \item  The center $\mathcal{Z}(X)$, equipped with the operator
            norm, is an AM-space with unit $I_X$. In particular,
            \[
            \|T\|=\|\, |T|\,\|=\inf \{\, \lambda >0 : |T|\le \lambda
            I_X \, \}
            \]
            for every $T \in \mathcal{Z}(X)$, and $\mathcal{Z}(X)$ is
            lattice isometric to $C(K)$ for a certain compact
            Hausdorff space $K$.
        \item  The center $\mathcal{Z}(X)$ is an inverse closed\footnote{A
                subalgebra $B$ of an algebra $A$ is said to be
                \emph{inverse closed} if for every $b \in B$ invertible in $A$,
            the inverse $b^{-1}$ is contained in $B$.}
            subalgebra of $\Lr X$.
        \item When $X$ is order complete, the center $\mathcal{Z}(X)$ is
            the band generated by $I_X$ in $\Lr X$. In particular, it
            is a projection band.
    \end{enumerate}
\end{thm}

An operator $T\colon X\to X$ on a Banach lattice $X$ is said to be
\begin{enumerate}
    \item an \emph{orthomorphism} if
$|x|\wedge |y|=0$ implies $|x|\wedge |Ty|=0$ for every $x,y \in X$;
    \item \emph{band preserving} if $T(B)\subseteq B$ for every
band $B$ in $X$.
\end{enumerate}

It is immediate that central operators are
orthomorphisms, and that orthomorphisms are band preserving. Not
trivial, but true, is that for Banach lattices these three classes of
operators coincide.

\begin{thm}[Abramovich--Veksler--Koldunov
    {\cite{abramovich_veksler_koldunov1979}}]\label{thm:central_orto_band}
    For an operator $T\colon X\to X$ on a Banach lattice $X$ the
    following are equivalent.
    \begin{enumerate}
        \item $T$ is a central operator.
        \item $T$ is an orthomorphism.
        \item $T$ is band preserving.
    \end{enumerate}
\end{thm}

\subsection{Complexification and spectra}\label{sec:regop_spectra}

We want to introduce the ``complex version'' of the spaces of regular
operators. But first we need to recall the notion of a complex
Banach lattice, following \cite[Section 2.2]{meyer-nieberg1991}. Let $X$ be a Banach lattice. The
\emph{complexification} of $X$, as a vector space, is the additive
group $X\times X$ with scalar multiplication
\[
    (\alpha +i\beta )(x,y)=(\alpha x-\beta y,\alpha y+\beta x),
\]
where $\alpha ,\beta \in \R$ and $x, y \in X$. The group $X\times X$ with this scalar multiplication is a complex vector
space, usually denoted by $X_\C$. Identifying $x \in X$ with $(x,0)
\in X\times X$ and $ix$ with $(0,x)$, we write $x+iy$ instead of
$(x,y)\in X\times X$.

\begin{prop}
    Let $X$ be a Banach lattice. For $z=x+iy \in X_\C$ define
    \[
    |z|=\sup \{\, x\cos\phi +y \sin\phi  : 0\le \phi \le 2\pi  \, \}.
    \]
    This supremum exists in $X$, and it satisfies:
    \begin{enumerate}
        \item $|z|=0$ if and only if $z=0$,
        \item $|\lambda z|=|\lambda| |z|$ for all $\lambda \in \C$ and
            $z \in X_\C$,
        \item $|z+w|\le |z|+|w|$ for all $z,w \in X_\C$.
    \end{enumerate}

    Moreover, the map $z\mapsto \| |z| \|$ defines a norm on $X_\C$,
    and $X_\C$ with this norm is a complex Banach space.
\end{prop}

\begin{defn}
    A \emph{complex Banach lattice} is a Banach space of the form
    $X_\C$, for a certain Banach lattice $X$, with the norm and
    modulus introduced in previous proposition.
\end{defn}

We now introduce regular operators between complex Banach lattices $X_\C$
and $Y_\C$. For every (complex) linear operator $T\colon X_\C\to Y_\C$
there exist unique (real) linear operators $T_1,T_2\colon X\to Y$ such
that $Tx=T_1x+iT_2x$ for all $x \in X$. The operator $T_1$ is called the \emph{real
part} of $T$, while the operator $T_2$ is called the \emph{imaginary part} of $T$.
Thus $\mathcal{L}(X_\C,Y_\C)$ is
isomorphic to the complexification of $\mathcal{L}(X,Y)$.
An operator $T\colon X_\C\to Y_\C$ is said to be
\emph{real valued} if $T(X)\subseteq Y$. If $T$ is real valued and
$T(X_+)\subseteq Y_+$, then $T$ is said to be \emph{positive}. An
operator $T\colon X_\C\to Y_\C$ is said to be \emph{regular} if both
its real and imaginary parts are regular. When $Y$ is order complete, the space of regular
operators between $X_\C$ and $Y_\C$ may be identified with the
complexification of the Banach lattice $\Lr{X,Y}$.

Given a regular operator $T\colon X_\C\to X_\C$, denote by $\sigma
(T)$ its spectrum. The following result
characterizes the lattice isomorphisms contained in the center in
terms of their spectrum, and will be needed later in the discussion.

\begin{thm}[Schaefer--Wolff--Arendt
    \cite{schaefer_wolff_arendt1978}]\label{thm:swa}
    Let $T$ be a lattice isomorphism of a complex Banach lattice $X$.
    Then $T \in \mathcal{Z}(X)$ if and only if $\sigma (T)\subseteq
    \R_+$.
\end{thm}

\subsection{Inner band projections}\label{sec:regop_innerbp}

For an order complete Banach lattice $X$, we have introduced the
Banach lattice $\Lr X$ of regular operators over $X$. We can then consider
band projections on this space. Since $\Lr X$ is again order complete,
any band will have an associated band projection. We have already
encountered an example of such a band in $\Lr X$, namely the center.
Another example are inner band projections, a particular
class of band projections on spaces of regular operators introduced by
the author together with P.\ Tradacete in
\cite{munoz-lahoz_tradacete2024}. The construction of these band
projections is summarized in the following result.

\begin{thm}[{\cite[Theorem 2.1]{munoz-lahoz_tradacete2024}}]\label{thm:proj_general}
    Let $X$ be an order complete Banach lattice and let $\{P_\lambda
    \}_{\lambda \in \Lambda }$ be a family of pairwise disjoint band
    projections on $X$. Let $\Gamma \subseteq \Lambda \times \Lambda $
    be any subset. Then the map
    \[\
        \begin{array}{cccc}
            \mathcal{P} _\Gamma \colon& \Lpos X & \longrightarrow & \Lpos X \\
            & T & \longmapsto & \displaystyle\bigvee_{(\alpha ,\beta)\in \Gamma } P_\alpha TP_\beta \\
        \end{array}
    \]
    extends to a unique band projection $\mathcal{P}_\Gamma \colon \Lr
    X\to \Lr X$.
\end{thm}

\begin{defn}
    Let $X$ be an order complete Banach lattice.  We say that a band
    projection $\mathcal{P}$ on $\Lr X$ is an \emph{inner band projection}
    if there exists a family of pairwise disjoint band projections
    $\{P_{\lambda }\}_{\lambda \in \Lambda }$ on $X$ such that
    $\mathcal{P}=\mathcal{P}_\Gamma $, as defined in
    \cref{thm:proj_general}, for a certain $\Gamma \subseteq \Lambda
    \times \Lambda $.  We will call \emph{inner projection bands} the
    bands associated with inner band projections, and we will denote
    $\mathcal{B}_\Gamma =\range(\mathcal{P}_\Gamma)$.
\end{defn}

In \cref{sec:bla_innerbp} we will be able to generalize this
construction to arbitrary order complete Banach lattice algebras.
Together with the construction, we will generalize some of the
properties of inner band projections.
In particular, we will study the following property, regarding the
structure of the set of inner band projections.

\begin{prop}[{\cite[Proposition
    2.10]{munoz-lahoz_tradacete2024}}]\label{prop:innerbp_boole}
    Let $X$ be an order complete Banach lattice and let $\{P_\lambda
    \}_{\lambda \in \Lambda }$ be a family of pairwise disjoint
    band projections on $X$. The inner band projections $\{
        \mathcal{P}_\Gamma : \Gamma  \subseteq \Lambda \times
    \Lambda  \} $ (resp.\ the inner projection bands $\{
    \mathcal{B}_\Gamma : \Gamma \subseteq \Lambda \times
    \Lambda  \} $) form a Boolean algebra with the same order,
    suprema and infima as in $\Lr {\Lr X}$. Moreover, this Boolean algebra is
    isomorphic to $2^{\Lambda \times \Lambda }$ through the map
    $\Gamma \mapsto \mathcal{P}_\Gamma $ (resp.\ $\Gamma \mapsto
    \mathcal{B}_\Gamma $).
\end{prop}

Inner band projections are also used to characterize
atomic and order continuous Banach lattices.

\begin{thm}[{\cite[Theorem
    3.1]{munoz-lahoz_tradacete2024}}]\label{thm:innerbp_all}
    Let $X$ be an order complete Banach lattice. All band projections
    on $\Lr X$ are inner if and only if $X$ is both atomic and order
    continuous.
\end{thm}

\section{Banach lattice algebras}\label{sec:bla}

The goal of this section is to generalize the
properties studied in the first section to Banach lattice algebras.
\Cref{sec:bla_intro,sec:bla_ideal} introduce Banach lattice algebras
and some known results that will be useful throughout.
The reader familiar with the literature can therefore have a quick
look at the notation in these sections, and then go directly to
\cref{sec:bpoi}.

\subsection{Banach lattice algebras}\label{sec:bla_intro}

In the space of regular operators over an order complete Banach
lattice, a Banach lattice structure (given by the cone of positive
operators) and a Banach algebra structure (given by composition)
coexist. We want to study more general spaces in which these two
structures coexist, and to extend (when possible) the properties of
the first
part of the paper to this abstract setting. This will elucidate which
properties of the space of regular operators do not really depend on
the fact that they are regular operators over a Banach lattice, but
rather on the algebraic and lattice structures of the space.

But first, we need to define Banach lattice algebras.
It seems like there is no consensus in the literature on the notion of
Banach lattice algebra with identity, so we follow the one suggested
by A.\ W.\ Wickstead in \cite{wickstead2017_questions}.

\begin{defn}
    A \emph{Banach lattice algebra} is a Banach lattice $A$ together
    with a product $\ast\colon A\times A\to A$ such that
    \begin{enumerate}
        \item $(A,\ast)$ is a Banach algebra,
        \item if $x,y \in A_+$, then $x\ast y \in A_+$.
    \end{enumerate}
    When the product $\ast$ has an identity $e$ of norm one, we say that
    $A$ is a \emph{Banach lattice algebra with identity $e$}. When the product
    $\ast$ is commutative, we say that $A$ is a \emph{commutative
    Banach lattice algebra}.
\end{defn}

It is important to note the following consequence of the definition.

\begin{prop}[{\cite[Theorem 2]{bernau_huijsmans1990_unit}}]
    The identity in a Banach lattice algebra is positive.
\end{prop}

Let us pause here to present some classical examples of Banach lattice
algebras. We shall resort to them when presenting examples and
counterexamples in our exposition.

\begin{example}
    \begin{enumerate}
        \item Let $K$ be a compact Hausdorff space. The space
            $C(K)$ of real-valued continuous functions on $K$ is a
            commutative Banach lattice algebra with pointwise order
            and product, and supremum norm. The identity is the constant one function.
        \item Let $X$ be an order complete Banach lattice. We have
            seen in \cref{sec:regop_intro} that the
            space of regular operators $\Lr X$ is a Banach lattice,
            and at the same time a Banach algebra, in which the
            composition of positive operators is positive. It is,
            therefore, a Banach lattice algebra, with
            identity $I_X$. This Banach lattice algebra need not be
            commutative.
        \item Let $G$ be a locally compact group, let $\lambda $ be a
            left invariant Haar measure, and consider the Banach
            lattice $L_1(G)=L_1(\lambda)$. The convolution
            multiplication on $L_1(G)$, defined for $\phi ,\psi \in
            L_1(G)$ by
            \[
                (\phi \ast \psi )(g)=\int_{G}^{} \phi (gh)\psi
                (h^{-1})\,d \lambda (h),\quad\text{where }g \in G,
            \]
            makes $L_1(G)$ into a Banach algebra. It is easy to check that
            the product of positive elements is positive, so $L_1(G)$
            is Banach lattice algebra. This Banach lattice
            algebra is commutative if and only if $G$ is abelian, and
            has an identity if and only if $G$ is discrete.
        \item Let $M(G)$ be the space of signed regular Borel measures
            on a locally compact group~$G$. Then $M(G)$ is a Banach
            lattice if we set $\mu \ge \nu $, $\mu ,\nu \in M(G)$,
            whenever $\mu
            (A)\ge \nu (A)$ for every Borel set $A$, and give it the
            norm $\|\mu \|=|\mu |(G)$. The convolution multiplication
            defined by
            \[
                (\mu \ast \nu )(A)=\int_G\int_{G} \chi _A(st)\,d \mu
                (s)\, d \nu (t),
            \]
            where $\chi _A$ denotes the characteristic function of the
            Borel set $A$, gives $M(G)$ a Banach algebra structure. Again, this
            space is a Banach lattice algebra. Its identity is the point mass measure on the identity of $G$.
        \item Let $\{A_i\}_{i \in I}$ be a family of Banach lattice
            algebras idexed by a set $I$, and let ${1\le p\le \infty}$.
            Then their $\el p$-sum $\el p(A_i)$, equipped with
            coordinate-wise order and
            product, is again a Banach lattice algebra, for if
            $x=(x_i)$ and $y=(y_i)$ are elements of $\el p(A_i)$, then
            \[
            \|xy\|_p\le \|x\|_p \|y\|_\infty \le \|x\|_p \|y\|_p.
            \]
            It is clear
            that $\el p(A_i)$ is commutative if and
            only if $A_i$ is commutative for each $i \in I$. Also,
            $\el \infty (A_i)$ has identity if and only if $A_i$ has
            identity for each $i \in I$.
    \end{enumerate}
\end{example}

\subsection{The ideal generated by the identity
element}\label{sec:bla_ideal}

As explained in \cref{sec:center}, the ideal generated by the identity
element in $\Lr X$ is the center of the Banach lattice $X$, an ideal of operators
extensively studied in the literature. It turns out that the
notion of the ideal generated by the identity element in Banach
lattice algebras has also been studied, and that properties analogous
to those in \cref{thm:center} hold in general.

If $A$ is a Banach lattice algebra with identity $e$, we adopt the
convention of \cite{scheffold1984} and denote
by $A_e$ the (order) ideal generated by $e$ in $A$. The main
properties of this ideal are collected in the following result.

\begin{thm}\label{thm:bla_ideal}
    Let $A$ be a Banach lattice algebra with identity $e$.
    \begin{enumerate}
        \item The space $(A_e, \|{\cdot }\|_e)$, with the order and
            product of $A$, is a Banach lattice
            algebra, lattice and algebra isometric to $C(K)$,
            for a certain compact Hausdorff $K$.\footnote{Note that
                the analogy with \cref{thm:center} is not complete:
                here we are endowing the ideal generated by the
                identity with its associated norm, whereas in
                \cref{thm:center} one works with the operator norm. As
                it turns out, assumption $\|e\|=1$ implies
                that both norms are equal. Yet we need spectra to show
                this and it will have to wait until
            \cref{sec:bla_spectra}. For the time being, it suffices to
        know that $A_e$ is closed in $A$. This apparent inconvenience
    has an advantage: all the contents of this section are valid
under the weaker hypothesis $e\ge 0$, which is another common
definition of Banach lattice algebras with identity.}
        \item The ideal $A_e$ is a principal projection band in $A$.
        \item The space $A_e$ is an inverse closed subalgebra of $A$.
    \end{enumerate}
\end{thm}

The contents of previous theorem are scattered through the literature:
the first assertion follows from the results by L.\ Martignon in \cite{martignon1980},
the second assertion is due to E.\ Scheffold \cite[Theorem 3]{scheffold1984}, while the third is
mentioned by C.\ B.\ Huijsmans in \cite{huijsmans1995} without any further reference.
Since this result will be central in our exposition, we devote the
rest of this section to present a full proof of it.

To prove the first part of the theorem, we need to talk about Banach lattice algebras
that are AM-spaces with unit, in which the order
unit coincides with the algebraic identity.

\begin{defn}
    Let $A$ be a Banach lattice algebra with identity $e$. If $A$ is also
    an AM-space with unit $e$ (i.e., if $A_e=A$ and $\|x\|=\inf \{\, \lambda >0 :
    |x|\le \lambda e\, \}$ for every $x \in A$), we say that $A$ is an
    \emph{AM-algebra with unit $e$}.
\end{defn}

The following theorem asserts that AM-algebras with unit are nothing
but algebras
of continuous functions.

\begin{thm}\label{thm:AMalg}
    Every AM-algebra with unit is algebra and lattice isometric to
    $C(K)$, for a certain compact Hausdorff $K$.
\end{thm}

Previous result is a direct consequence of Kakutani's representation
theorem for AM-spaces with unit, and the following
characterization of the pointwise product in $C(K)$ due to L.\
Martignon.

\begin{prop}[{\cite[Proposition 1.4]{martignon1980}}]
    Let $K$ be a compact Hausdorff space and let $\star\colon
    C(K)\times C(K)\to C(K)$ be a binary operation such that
    $(C(K),\star)$
    with pointwise order is a Banach lattice algebra.
    Suppose also that the constant one function is the identity for
    $\star$. Then $\star$ is the pointwise multiplication.
\end{prop}

The proof of this last result relies heavily on Banach lattice techniques.
Here we want to provide an alternative proof of \cref{thm:AMalg} in the style
of the Gelfand transform: working with maximal (order and algebraic)
ideals and real-valued (lattice and algebra) homomorphisms. This
argument follows closely the proof of Kakutani's representation
theorem from \cite[Chapter 4]{troitsky}.

There are two kinds of ideals present in Banach lattice algebras: the
ideals as a Banach lattice and the ideals as a Banach algebra. There
are also subspaces that are ideals in both senses. There does not seem
to be a common terminology for these spaces in the literature. We will
refer to ideals in the Banach lattice sense as
\emph{order ideals}, and to ideals in the Banach algebra sense as
\emph{algebraic ideals} (these, in turn, can be left or right ideals,
or both). \emph{Ideals} will be the subspaces that are
at the same time an order ideal and a (two-sided) algebraic ideal. The
following is a simple characterization of the ideal generated
by an element (that is, the intersection of all the ideals containing
the element) in a Banach lattice algebra with identity.

\begin{lem}\label{lem:bla_principal_id}
    Let $A$ be a Banach lattice algebra with identity and let $a \in A$. Then
    the ideal generated by $a$ is
    \[
    \{\, x \in A : |x|\le r|a|r' \text{ for some }r,r' \in A_+ \, \}.
    \]
\end{lem}
\begin{proof}
    Call the displayed set $I$. Let us check that this set is indeed
    an ideal. If $x,y \in I$, say $|x|\le rar'$ and $|y|\le sas'$ for some
    $r,r',s,s' \in A_+$, then
    \[
    |x+y|\le |x|+|y|\le rar'+sas'\le (r+s)a(r'+s')
    \]
    so $x+y \in I$. Clearly, $-x$ and $|x|$ are elements of $I$, so $I$ is a
    sublattice. From the definition it follows that $I$ is an
    order ideal. To check that it is also an algebraic ideal, let $z
    \in A$. Then we have
    \[
    |zx|\le |z||x|\le (|z|r)ar'\quad\text{and}\quad|xz|\le |x| |z|\le
    ra(r'|z|)
    \]
    so both $zx, xz \in I$, and we have thus shown that $I$ is an
    ideal.

    Since $A$ has an identity, $a \in I$.
    Let $J$ be any ideal of $A$ containing $a$. Then, since $J$ is a
    sublattice and an
    algebraic ideal, $r|a|r' \in J$ for all $r,r' \in A_+$. Since
    $J$ is an order ideal, $x \in J$ whenever $|x|\le rar'$. Hence
    $I\subseteq J$, and so $I$ is the ideal generated by $a$ in $A$.
\end{proof}

As usual, we say that an ideal in a Banach lattice algebra is
\emph{maximal} if it is not contained in any proper ideal.
The following is pure routine.

\begin{lem}\label{lem:bla_max_id}
    In a Banach lattice algebra with identity, any proper ideal is contained in a maximal ideal.
\end{lem}
\begin{proof}
    Let $I$ be a proper ideal in a Banach lattice algebra $A$ with
    identity $e$. Let $\mathcal{I}$ be the
    family of proper ideals in $A$ containing $I$, ordered by
    inclusion. Let $\{J_\gamma \}_{\gamma \in \Gamma }$ be a chain in
    $\mathcal{I}$.
    Then it is straightforward to check that $J=\bigcup_{\gamma \in
    \Gamma } J_\gamma $ is again an ideal containing $I$. Moreover, if
    $J=A$, then $e \in J$ and so $e \in J_\gamma $ for a certain
    $\gamma \in \Gamma $. Being $J_\gamma $ an algebraic ideal, this
    would imply $J_\gamma =A$, contradicting the definition. So $J$ is
    proper, and therefore is an upper bound of the chain in $\mathcal{I}$. By Zorn's
    lemma, there exists a maximal element in $\mathcal{I}$, and this
    element is a maximal ideal containing $I$.
\end{proof}

Let $A$ be a Banach lattice algebra and let $I$ be a closed ideal. Perform
the usual quotient $A/I$. Being $I$ a two-sided closed ideal, $A/I$ is
again a Banach algebra, and the quotient map $Q\colon A\to
A/I$ is an algebra homomorphism.
Since $I$ is an order ideal, $Q(A_+)$ is a cone in $A/I$, and the
order it induces makes $A/I$ into a Banach lattice, and $Q$ into a
lattice homomorphism. It is direct to check that the product of
positive elements in $A/I$ is again positive. Hence $A/I$ admits a
natural Banach lattice algebra structure, making $Q$ a lattice and
algebra homomorphism.

Since the closure of an algebraic ideal is again an
algebraic ideal, and the closure of an order ideal is again an order
ideal, the closure of an ideal in a Banach lattice algebra is again an
ideal. It is then routine to check that maximal ideals in
Banach lattice algebras with identity are closed. Moreover, we have the following.

\begin{lem}\label{lem:bla_quot_max}
    Let $A$ be an AM-algebra with unit $e$ and let $M$ be a maximal ideal
    in $A$. Then $A/M$ is lattice and algebra isomorphic to $\R$.
\end{lem}
\begin{proof}
    By the discussion preceding the lemma, $A/M$ is a Banach lattice
    algebra. Suppose that $\dim(A/M)>1$. Then there would
    exist two non-zero disjoint elements $Qx, Qy \in A/M$. We are
    going to show that the ideal $I$ generated by $|Qx|$ in $A/M$
    must be proper. Indeed, if $Qy \in I$, by
    \cref{lem:bla_principal_id} there would exist $r,r' \in A_+$ such
    that $|Qy|\le Q(r)|Q(x)|Q(r')$, and then
    \[
    |Qy|= (Q(r)|Q(x)|Q(r'))\wedge
    |Qy|\le \|r\| \|r'\| |Q(x)|\wedge |Q(y)|=0,
    \]
    which is a contradiction. Hence $I$ is proper. It is not difficult
    to check that $Q^{-1}(I)\subseteq A$ is again an ideal, because
    $Q$ is both a lattice and algebra homomorphism. Moreover,
    $M\subseteq Q^{-1}(I)\subsetneq A$, and since $x \in Q^{-1}(I)$,
    the first inclusion is proper. This contradicts that $M$ is
    maximal. So $\dim(A/M)$ must be $1$. Then $A/M$ is spanned by
    $Qe$ and the map
    \[
    \begin{array}{cccc}
    & \R & \longrightarrow & A/M \\
            & \lambda  & \longmapsto & \lambda Qe \\
    \end{array}
    \]
    is a lattice and algebra isomorphism.
\end{proof}

\begin{lem}\label{lem:bla_functionals}
    Let $A$ be an AM-algebra with unit $e$ and let $x_0 \in A_+$. If
    $x_0\not\le e$, there exists a lattice and algebra homomorphism
    $\phi \colon A\to \R$ such that $\phi (x_0)>1$.
\end{lem}
\begin{proof}
    We have $(x_0-e)_+\neq 0$. Let $I$ be the ideal generated by
    $(x_0-e)_-$. We claim that $(x_0-e)_+ \not\in I$. Indeed, if it
    were the case, according to \cref{lem:bla_principal_id} there would exist $r,r' \in A_+$ such that
    \[
        (x_0-e)_+\le r(x_0-e)_-r'\le \|r\|\|r'\|(x_0-e)_-
    \]
    which contradicts the fact that $(x_0-e)_+$ and $(x_0-e)_-$ are
    disjoint. Hence $I$ is
    proper, and by \cref{lem:bla_max_id} it is contained in a maximal
    ideal $M$. The quotient map associated
    with $M$ is a lattice and algebra homomorphism $A\to
    A/M$. Let $\phi\colon A\to \R$ be the composition of this quotient map with the
    isomorphism between $A/M$ and $\R$ given in
    \cref{lem:bla_quot_max}. This homomorphism satisfies
    \[
    \phi (x_0-e)=\phi ((x_0-e)_+)> 0,
    \]
    so $\phi (x_0)>\phi (e)=1$.
\end{proof}

With this set of rather standard tools, we can proceed to the proof of
\cref{thm:AMalg}.

\begin{proof}[Proof of \cref{thm:AMalg}]
    Let $A$ be an AM-algebra with unit $e$. Let $K$ be the set of all
    lattice and algebra
    homomorphisms $\phi \colon A\to \R$ with $\phi (e)=1$. We can view $K$ as a subset
    of $\R^{A}$. Equip $\R^{A}$ with the product topology and $K$ with the
    subspace topology. This topology is Hausdorff, because the
    product topology is. Note that, for every $\phi \in K$,
    $|\phi (x)|\le \phi (\|x\| e)=\|x\|$ holds, which
    means that $K\subseteq \prod_{x \in A} [-\|x\|,\|x\|]$. By
    Tychonoff's theorem, the latter product is compact. Moreover, $K$
    is closed in this subspace, for if $(\phi _\lambda )$ is a net in
    $K$ pointwise convergent to $\phi $, this means that $\phi
    (x)=\lim \phi _\lambda (x)$ for all $x \in A$, and then by
    continuity of the algebraic and lattice operations it is easy
    to check that
    $\phi $ is a again a lattice and algebra homomorphism with $\phi
    (e)=1$. Therefore $K$ is a compact Hausdorff space.

    It is useful to note that $K$ has sufficiently many elements.
    Indeed, for every $x \in A_+$ with $\|x\|>1$, we have $x\not \le e$,
    and by \cref{lem:bla_functionals} there is some element $\phi \in K$
    such that $\phi (x)>1$.

    Now consider the map $T\colon A\to C(K)$ where $Tx(\phi )=\phi
    (x)$ for $x \in A$ and $\phi \in K$. It is easy to check that
    $Tx$ is indeed continuous. Since
    operations in $C(K)$ are computed pointwise, and $Te(\phi )=\phi
    (e)=1$ for all $\phi \in K$, it
    follows that $T$ is a lattice and algebra homomorphism that
    preserves the identity.

    We claim $\|Tx\|=\|x\|$ for all $x \in A$. On one hand,
    $|Tx(\phi )|=|\phi (x)|\le \|x\|$, so $\|Tx\|\le \|x\|$. On the other
    hand, if $\|x\|>1$, we have seen that there exists some $\phi \in
    K$ such that $\phi (|x|)> 1$, so that $\|Tx\|\ge |Tx(\phi
    )|=\phi (|x|)> 1$. If $\|Tx\|<\|x\|$ for some $x \in A$, the
    element $|x|/\|Tx\|$ would contradict the previous observation. So
    $\|Tx\|=\|x\|$ for all $x \in A$.

    Finally, note that $T(A)$ is a closed
    subalgebra of $C(K)$ that contains the constants
    and separates points: if $\phi \neq \psi  $ in $K$, there is some
    $x \in A$ such that $(Tx)(\phi )=\phi (x)\neq \psi (x)=(Tx)(\psi
    )$. By the Stone--Weierstrass theorem, $T(A)=C(K)$, and the result
    is proved.
\end{proof}

Some preliminary remarks before presenting the proof of \cref{thm:bla_ideal}. For the second
assertion we will provide an argument similar to
\cite[Theorems 1 and 2]{huijsmans1988} instead of
the original one appearing in \cite{scheffold1984} because it serves
as a first illustration of what will be
our main trick in the coming sections, namely, to consider multiplication
operators on a Banach lattice algebra, and then exploit the
known properties of the space of regular operators. On a Banach lattice
algebra $A$, operators of the form
\[
\begin{array}{cccc}
L_a\colon& A & \longrightarrow & A \\
        & x & \longmapsto & ax \\
\end{array},\quad
\begin{array}{cccc}
R_b\colon& A & \longrightarrow & A \\
        & x & \longmapsto & xb \\
\end{array},
\]
for certain $a,b \in A$, are called \emph{left} and \emph{right multiplication
operators}, respectively. Operators of the form $L_aR_b$, for some
$a,b \in A$, are called \emph{multiplication operators}.

\begin{proof}[Proof of \cref{thm:bla_ideal}]
    \begin{enumerate}
        \item We are going to check that $(A_e,\|{\cdot }\|_e)$ is an
            AM-algebra with unit $e$. Then the result will follow from
            \cref{thm:AMalg}. That
            $(A_e,\|{\cdot }\|_e)$ is an AM-space with unit $e$ is
            clear. Moreover, it is a subalgebra, for if $x,y \in A_e$:
            \[
            |xy|\le |x| |y|\le \|x\|_e \|y\|_e e.
            \]
            From this also follows that $\|xy\|_e\le \|x\|_e \|y\|_e$,
            so it is a Banach algebra. Its algebraic identity is
            obviously $e$.
        \item We prove this fact in two steps. First, we show that
            $B_e$, the band generated by $e$, is a projection band.
            And then we show that $A_e=B_e$.

            To prove that $B_e$ is a projection band, it suffices to
            check that $\sup_n a\wedge ne$ exists for every $a \in A_+$
            (in fact, we show a little more, because we will need it
            later on). Denote $a_n=a\wedge ne$, and observe that for
            any $m\ge n\ge 1$:
            \[
                (a_m-a_n)\wedge (ne-a_n)=(a_m\wedge ne)-a_n=a_n-a_n=0.
            \]
            Using the representation of $I_e$ as a $C(K)$ space, and
            seeing $a_m$, $a_n$ and $ne$ as continuous functions in
            $K$, we are saying that $(a_m-a_n)(t)\wedge (ne-a_n)(t)=0
            $ for every $t \in K$. It follows that $a_m -a_n$ and $a_m-n
            ^{-1}a_ma_n=n^{-1}a_m(ne-a_n)$ are also disjoint, so
            \[
                (a_m-a_n)\wedge (a_m-n^{-1}a_ma_n)=0
            \]
            which implies
            \[
            a_m=a_n\vee (n^{-1}a_ma_n).
            \]
            It follows that
            \begin{equation}\label{eq:uniform_conv}
            0\le \|a_m-a_n\|=\|(n^{-1}a_ma_n-a_n)_+\|\le
            n^{-1}\|a_ma_n\|\le n^{-1}\|a\|^2,
            \end{equation}
            so $(a_n)$ is a Cauchy sequence. It therefore converges  to
            some $b \in A_+$. Since $(a_n)$ is increasing,
            we must have $b=\sup_n a\wedge ne$, and this shows that
            $B_e$ is a projection band.

            To prove that $A_e$ is a band we need to show that $A_e=B_e$, and for this it
            suffices to check that $a \in A_e$ whenever $a \in
            (B_e)_+$. For every $a \in (B_e)_+$, inequality $a\wedge ne
            \le ne$ implies that $L_{a\wedge ne}$ is a central
            operator for every $n\ge 1$. In particular, it is an orthomorphism. If $x,y
            \in A$ are such that $|x|\wedge |y|=0$, then $|x|\wedge
            |(a\wedge ne)y|=0$ for every $n\ge 1$. But since $a$ is an
            element of the band generated by $e$, we have
            $a=\sup_n a\wedge ne$. Then $a\wedge ne$ converges to $a$
            by \eqref{eq:uniform_conv}. Since both multiplication and
            lattice operations are continuous for the norm, it follows
            \[
            0=\lim_{n \to \infty}|x|\wedge |(a\wedge ne)y|=|x|\wedge
            |ay|.
            \]
            This shows that $L_a$ is an orthomorphism. By
            \cref{thm:central_orto_band},
            $L_a$ must be a central operator, that is, there exists
            some $\lambda \in \R_+$ such that $-\lambda I_A\le L_a\le
            \lambda I_A$. In particular, $a=L_a(e)\le \lambda e$, and
            $a \in A_e$ as wanted.
        \item Let $a \in A_e$ and suppose that $a^{-1}$ exists in $A$.
            By the previous assertion we can decompose
            $a^{-1}=a_1+a_2$ with $a_1 \in A_e$ and $a_2\in A_e^{d}$.
            Then $e=a a^{-1}=a a_1+a a_2$. Since $A_e$ is a
            subalgebra, $a a_1\in A_e$. Also
            \[
            0\le |a a_2|\wedge e\le (|a| |a_2|)\wedge e\le \|a\|_e
            |a_2|\wedge e=0,
            \]
            so $a a_2\in A_e^{d}$. But $a a_2=e- a a_1\in A_e$, so it
            must be $a a_2=0$. Multiplying by $a ^{-1}$ we get $a_2=0$
            and $a^{-1}=a_1 \in A_e$, as wanted.\qedhere
    \end{enumerate}
\end{proof}

\begin{example}\label{ex:ALalg}
    Let $A$ be a Banach lattice algebra with identity $e$ that
    is an AL-space as a Banach lattice. Then $(A_e,\|{\cdot
    }\|_e)$ is an AM-space, but at the same time is closed, so
    that $(A_e,\|{\cdot }\|_e)$ is also isomorphic to an
    AL-space. It follows that $A_e$ must be
    finite-dimensional.
    In fact, it will follow from our study of
    spectra that, under the assumption $\|e\|=1$, $\|{\cdot
    }\|=\|{\cdot }\|_e$ (see \cref{thm:bla_ideal_complex}). This
    implies $A_e=\R$. Without the assumption $\|e\|=1$, however, $A_e$
    need not be one-dimensional: just take $\el 1^{n}$ with the
    coordinate-wise order and product.
\end{example}

\subsection{Band projections and order idempotents}\label{sec:bpoi}

Band projections play a prominent role in the theory of Banach
lattices. These operators are characterized among
regular operators by \cref{thm:regop_bpchar}. The goal of this section
is to define elements in a general Banach lattice algebra that play
the same role as band projections do in spaces of regular operators. When
an identity is present, this generalization is straightforward
from \cref{thm:regop_bpchar}, and has already been
considered by E.\ A.\ Alekhno in the context of ordered Banach
algebras (see \cite{alekhno2018,alekhno2012}).

\begin{defn}
    Let $A$ be a Banach lattice algebra with identity $e$. We say
    that $p \in A$ is \emph{order idempotent} if $p^2=p$ and $0\le
    p\le e$. We denote by $OI(A)$ the set of order idempotents in $A$.
\end{defn}

What follows is the generalization of \cref{thm:regop_boole}. This
result appears, with a different proof, in \cite[Corollary
2.2]{alekhno2012}.

\begin{prop}
    Let $A$ be a Banach lattice algebra with identity $e$. The set of
    order idempotents $OI(A)$ forms a Boolean algebra with operations
    \[
    \overline{p}=e-p,\ p\vee q=p+q-pq,\ p\wedge q=pq,\quad \text{for } p,q \in OI(A),
    \]
    with zero as the minimum and the identity $e$ as the maximum. Moreover,
    suprema and infima coincide with those in $A$.
\end{prop}
\begin{proof}
    By definition, $OI(A)$ is contained in $A_e$. Represent
    $A_e$ as a $C(K)$ space using \cref{thm:bla_ideal}. Then $p \in
    OI(A)$, seen
    as a continuous function in $K$, satisfies $p^2=p$; that is, $p$
    is a characteristic function of a clopen set in $K$. Conversely, any
    characteristic function of a clopen set is an order idempotent in
    $C(K)$, and therefore in $A$. This exhibits an identification of
    $OI(A)$ with $\{\, \chi _C : C\subseteq K\text{ clopen} \, \}.$
    The result now clearly follows.
\end{proof}

\begin{prop}
    In a Banach lattice algebra with identity,
    the set of order idempotents is norm and order closed.
\end{prop}
\begin{proof}
    Let $A$ be a Banach lattice algebra with identity $e$. Let $(p_n
    )$ be a sequence of order idempotents converging to $p$ in
    norm. Then $p_n ^2=p_n $ and $0\le p_n \le e$ for
    each $n \in \N $, and since the product is continuous and the
    positive cone is closed, it follows that $p$ is an order
    idempotent. This shows that the set of order idempotents is norm
    closed.

    To show that it is also order closed, let $(p_\alpha )$ be a net
    of order idempotents converging to $p$ in order. The product need
    not be order continuous, so one has to proceed with more care now.
    Since the positive cone is always order closed, $0\le p\le
    e$ follows from $0\le p_\alpha \le e$ for each $\alpha $. By
    \cref{thm:bla_ideal}, $p$ and $p_\alpha $ commute, so
    \[
    |p^2-p_\alpha^2 |\le (p+p_\alpha )|p-p_\alpha |\le 2|p-p_\alpha |,
    \]
    where the first inequality follows from the submultiplicativity of
    the modulus. This implies that $p_\alpha ^2$ converges in order to
    $p^2$. Since also $p_\alpha ^2=p_\alpha $ for every $\alpha
    $, and the order limit is unique, $p^2=p$. So $p$ is an order
    idempotent, and this shows that the set of order idempotents is
    also order closed.
\end{proof}

The notion of order idempotent seems to abstract some of the
most important properties of band projections, but it relies on the
existence of an identity element. Of course, this poses no limitation in
spaces of regular operators, where the identity operator is always
present. However, this is not the case for general Banach lattice algebras. Can we
somehow still abstract the
definition and properties of band projections to Banach
lattice algebras without identity?

Faced with a similar problem, K.\ Vala gives in \cite{vala1967} a notion of compact and finite rank elements in
arbitrary normed algebras as follows:
for a normed algebra $A$, he defines an element $a$ in $A$ to be
\emph{compact} (resp.\ \emph{of finite rank}) if the operator
$L_aR_a$ is compact (resp.\ of finite rank).  In view of
\cref{cor:bp_multop}, it makes sense to take Vala's approach and set
the following definition.

\begin{defn}
    Let $A$ be a Banach lattice algebra. An element $a \in A_+$ is
    called a
    \emph{band projection} if $L_aR_a\colon A\to A$ is a band
    projection. The set of band projections in $A$ is denoted
    \[
    BP(A)=\{\, a \in A_+ : L_aR_a\text{ is a band projection} \, \}.
    \]
\end{defn}

It is also convenient to have the following notions.

\begin{defn}
    Let $A$ be a Banach lattice algebra. An element $a \in A_+$ is
    called a \emph{left band projection} (resp.\ \emph{right band
    projection}) if $L_a$ (resp.\ $R_a$) is a band projection. The
    sets of left and right band projections in $A$ are denoted
    \[
    BP_l(A)=\{\, a \in A_+ : L_a\text{ is a band projection} \, \},
    \]
    \[
    BP_r(A)=\{\, a \in A_+ : R_a\text{ is a band projection} \, \}.
    \]
\end{defn}

\begin{rem}
    Let $A$ be a Banach lattice algebra. According to
    \cref{thm:regop_bpchar},
    $a \in A_+$ is a band projection if and only if
    \[
    0\le axa\le x\text{ and }a^2xa^2=axa\quad\text{for all }x \in A_+.
    \]
    Similarly, $a \in A_+$ is a left band projection if and only if
    \[
    0\le ax\le x\text{ and }a^2x=ax\quad\text{for all }x \in A_+,
    \]
    and a right band projection if and only if
    \[
    0\le xa\le x\text{ and }xa^2=xa\quad\text{for all }x \in A_+.
    \]
    In particular, the sets $BP(A)$, $BP_l(A)$ and $BP_r(A)$ are all
    norm closed.
\end{rem}

If the multiplication in $A$ is jointly order continuous, it follows
from previous remark that $BP(A)$, $BP_l(A)$ and $BP_r(A)$ are
order closed. In general, however, this will not be true, as the
following example shows.

\begin{example}\label{ex:pathologicalc}
    Consider the Banach lattice $c$ of convergent real sequences, with
    the product defined by
    \[
    x\star y=\big(\lim_{n\to \infty }x_n\big)\big(\lim_{n \to
        \infty}y_n\big)e \quad
    \text{where }e=(1,1,1,\ldots ).
    \]
    Associativity of this product follows from
    \[
        \lim_{n\to \infty } (x\star y)_n=\big(\lim_{n\to \infty }
        x_n\big)\big(\lim_{n \to \infty } y_n\big),
    \]
    bilinearity is clear, and it is not difficult to check that, under
    the usual supremum norm, $c$ with $\star$ is a Banach
    algebra.
    Moreover, the product of positive elements is positive, so $c$ is
    actually a Banach lattice algebra (without identity).

    It is clear that $c_0=BP(c)=BP_l(c)=BP_r(c)$.
    In particular, $x_n=\sum_{i=1}^{n}e_i$, where $e_i$ denotes the $i$th
    coordinate vector, is a (left and right) band projection. Also,
    $x_n\uparrow e$. However, $e$ is not a (left or right) band projection, because
    for $x=e-e_1$ we have
    \[
    e\star x\star e=e\not\le x.
    \]
\end{example}

It is natural to ask what is the relation between the sets $BP(A)$,
$BP_l(A)$, $BP_r(A)$ and, in the case with identity, $OI(A)$. From the
definition, it should be clear that $BP_l(A)\cap BP_r(A)\subseteq
BP(A)$, and that, in the case with identity, $BP_l(A)=OI(A)=BP_r(A)$ (for $p
\in BP_l(A)$, just evaluate $L_p$ at the identity element, and similarly
for $p \in BP_r(A)$). In particular, $OI(A)=BP_l(A)\cap
BP_r(A)\subseteq BP(A)$. Is it true that $OI(A)=BP(A)$? This seems to
be the case at least for the classical examples of Banach lattice
algebras.

\begin{example}
    \begin{enumerate}
        \item Let $X$ be an order complete Banach lattice. If $P \in
            BP(\Lr X)$, by \cref{cor:bp_multop}, $P$ is a band
            projection, that is, $P \in OI(\Lr X)$.
        \item Let $f \in BP(C(K))$, for a certain compact Hausdorff
            space $K$, and let $\mathbb{1}$ be the constant one
            function in $C(K)$. Then
            $L_fR_f(\mathbb{1})=(L_fR_f)^2(\mathbb{1})$ implies
            $f^2=f^{4}$, so $f$ must be the characteristic function of a
            clopen set. It follows that $f$ is an order idempotent.
        \item Let $G$ be a locally compact group and consider the
            Banach lattice algebra $M(G)$. Since the identity $\delta _e$
            is an atom, if $\mu \in OI(M(G))$, inequality $0\le
            \mu \le \delta _e$ implies $\mu =\lambda \delta
            _e$ for some $\lambda \ge 0$. Then $\mu ^2=\mu $ (we
            denote $\mu ^2=\mu \ast \mu $) forces either
            $\lambda =1$ or $\lambda =0$. Hence
            $OI(M(G))=\{0,\delta _e\}$ is trivial.

            We are going to show that $BP(M(G))$ is also trivial. If
            $\mu \in BP(M(G))$, evaluating $L_\mu R_\mu $ at $\delta
            _e$ yields that $\mu ^2$ must be an order idempotent.
            That is, either $\mu ^2=0$ or $\mu ^2=\delta _e$. By direct
            computation
            \[
            \mu ^2(G)=\iint_{G\times G} \chi _G(st)\, d \mu (s) d \mu
            (t)=\int_{G}^{} \mu (s ^{-1}G)\, d \mu (s)=\mu (G)^2,
            \]
            so if $\mu ^2=0$, $\mu (G)^2=0$, and since $\mu\ge 0$,
            this implies $\mu =0$. Now suppose $\mu\neq 0$, that is,
            $\mu ^2=\delta _e$. We can decompose $\mu =\mu _d+\mu _c$,
            where $\mu _d\ge 0$ is discrete (i.e., is contained in the
            band generated by the atomic measures, which is precisely
            the closed linear span of said measures, because $M(G)$ is order
            continuous) and $\mu _c\ge 0$ is
            continuous (i.e., is contained in the disjoint complement).
            Since
            \[
                \mu _c^2(\{g\})=\int_{G}^{} \mu _c(\{s^{-1}g\})\, d
                \mu _c(s)=0\quad\text{for every }g \in G,
            \]
            the measure $\mu _c^2$ cannot have any discrete component.
            However,
            \[
            \delta _e=\mu ^2=\mu _d^2+\mu _c^2+\mu _d\ast\mu _c+\mu
            _c\ast \mu _d
            \]
            and since every measure involved is positive, we must have
            $\mu _c^2=0$. Equality $\mu _c(G)^2=\mu _c^2(G)=0$ implies
            $\mu _c=0$ and so $\mu $ is discrete. Write
            \[
            \mu =\sum_{g \in G}^{} \lambda _g \delta _g
            \]
            for some $\lambda _g \in \R_+$. By definition,
            \[
            \delta _e=\mu ^2=\sum_{g \in G}^{} \bigg( \sum_{h \in G}^{}\lambda
            _h \lambda _{h^{-1}g} \bigg) \delta _g,
            \]
            and since the $\lambda _h$ are positive, this can only
            happen if $\lambda _h=0$ for all $h\neq e$. Therefore $\mu
            =\lambda _e \delta _e$, and from $\mu ^{2}=\mu ^4$ it
            follows that $\lambda _e=1$. We have thus shown that if
            $\mu  \in BP(M(G))$ is different from $0$, it must be
            $\delta _e$, as wanted.

        \item A.\ W.\ Wickstead provided in \cite{wickstead2017_two} an
            explicit list of all two-dimensional Banach lattice
            algebras. In the Banach lattice algebras with identity
            from the list, one can check, by means of direct computation, that the order
            idempotents coincide with the band projections.
    \end{enumerate}
\end{example}

Contrary to this evidence, it is not true that in Banach lattice algebras with identity all
band projections are order idempotents, and an example already exists
in dimension 3.

\begin{example}\label{ex:oinotbp}
    Consider $\R^{2}$ with the usual lattice structure and any lattice
    norm, and let $\Lr{\R^2}$ be its space of
    regular operators. The space $\Lr{\R^2}$ is a Banach lattice
    algebra, which we may identify with $M_2(\R)$, the space of
    $2\times 2$ real matrices, where the
    order is defined entry-wise and the product is the composition
    (that is, the usual matrix product).
    Consider the set of upper triangular matrices
    \begin{align*}
        A&=\barespn\left\{
        \begin{pmatrix}
            1 & 0 \\ 0 & 0
        \end{pmatrix},
        \begin{pmatrix}
            0 & 1 \\ 0 & 0
        \end{pmatrix},
        \begin{pmatrix}
            0 & 0 \\ 0 & 1
        \end{pmatrix}
    \right\}\\
         &=\left\{\, \begin{pmatrix} a & b \\ 0 & c \end{pmatrix}
         :  \, a,b,c \in \R\right\}.
    \end{align*}
    Being spanned by disjoint elements, $A$ is a 3-dimensional
    sublattice of $M_2(\R)$. Moreover, one can easily check that the product of
    upper-triangular matrices is again upper-triangular. Hence, $A$ is a
    Banach lattice algebra with identity when equipped with the order,
    product, and norm inherited from $M_2(\R)$.

    Now consider $p=\begin{pmatrix} 0 & 1 \\ 0 & 0 \end{pmatrix} $.
    Clearly $p\ge 0$ and, for all $a,b,c \in \R$,
    \[
        L_pR_p \begin{pmatrix} a & b \\ 0 & c \end{pmatrix} =
        \begin{pmatrix} 0 & 0 \\ 0 & 0 \end{pmatrix}.
    \]
    Hence $L_pR_p=0$; in particular, $p \in BP(A)$. However, $p\not\le I_2$,
    so $p \not\in OI(A)$.

    Let us compute exactly $BP(A)$ and $OI(A)$. It is direct to check
    that
    \[
        OI(A)=\left\{\begin{pmatrix} 0 & 0 \\ 0 & 0 \end{pmatrix}
        ,\begin{pmatrix} 1&0\\0&0 \end{pmatrix} ,\begin{pmatrix}
    0&0\\0&1 \end{pmatrix} ,\begin{pmatrix} 1&0\\0&1 \end{pmatrix}
\right\}.
    \]
    Now let $\bar p = \begin{pmatrix} p_1&p_2\\0&p_3 \end{pmatrix} \in
    BP(A)$. For any $a,b,c \in \R_+$:
    \[
        \begin{pmatrix} p_1&p_2\\0&p_3 \end{pmatrix} \begin{pmatrix}
        a&b\\0&c \end{pmatrix} \begin{pmatrix} p_1&p_2\\0&p_3
        \end{pmatrix} =\begin{pmatrix}
    p_1^2a&p_1p_2a+p_1p_3b+p_2p_3c\\0&cp_3^2 \end{pmatrix} \le
    \begin{pmatrix} a&b\\0&c \end{pmatrix}.
    \]
    Setting $b=0$ we get the inequality in the upper right entry
    \[
    p_2(ap_1+cp_3)\le 0.
    \]
    Since $a$ and $c$ can take any positive value, and $p_1,p_2,p_3$
    are positive, either $p_2=0$
    or $p_1=p_3=0$.
    \begin{enumerate}
        \item If $p_2=0$, equality ${\bar p}^2={\bar p}^{4}$ implies
            $p_1^2=p_1^{4}$, so $p_1=0$ or $p_1=1$, and similarly
            $p_3=0$ or $p_3=1$. The only band projections in this case are the
            order idempotents.
        \item If $p_2\neq 0$, then $p_1=p_3=0$, and $\bar{p}$ is a
            band projection for any $p_2\ge 0$.
    \end{enumerate}
    Putting together these computations:
    \[
        BP(A)=OI(A)\cup \left\{ \begin{pmatrix} 0&p\\0&0 \end{pmatrix}  :
    p \in \R_+ \right\}.    \]
\end{example}

\begin{rem}
    This example already illustrates a difference between
    $BP(A)$ and $OI(A)$: while order idempotents commute by
    \cref{thm:bla_ideal}, band
    projections need not commute.
    After this observation, it is not surprising that $BP(A)$ need not be a Boolean algebra:
    in previous example, $BP(A)$ is not closed under taking suprema.

    By contrast, the elements of $BP_l(A)\cap BP_r(A)$ always commute. Indeed, let
    $a,b \in BP_l(A)\cap BP_r(A)$. Since $L_a, L_b, R_a, R_b$ are band
    projections, they commute, so $aba=L_aL_b(a)=L_bL_a(a)=ba^2$ and
    $aba=R_aR_b(a)=R_bR_a(a)=a^2b$, which gives
    $ab=a^2b=aba=ba^2=ba$.
\end{rem}

One could think that previous example is somewhat pathological,
because $p^2=0$ and $L_pR_p=0$. Maybe this is the only reason a band
projection is not an order idempotent. It is not, and a
counterexample is coming. But first an observation.

\begin{rem}
    Let $\{A_i\}_{i \in I}$ be a family of Banach lattice algebras
    indexed on a set $I$. Consider $A=\el p(A_i)$ for some $1\le p\le
    \infty $. We claim that
    \[
    BP(\el p(A_i))=\prod_{i \in I} BP(A_i).
    \]
    Indeed, a positive element $(p_i)_{i \in I}$ is a band
    projection in $\el p(A_i)$ if and only if for all
    $(x_i)_{i \in I} \in \el p(A_i)_+$:
    \[
    0\le (p_i)(x_i)(p_i)\le
    (x_i)\quad\text{and}\quad
        (p_i)(x_i)(p_i)=(p_i)^2(x_i)(p_i)^2.
    \]
    This is the same as saying that $0\le p_ix_ip_i\le x_i$ and $p_i
    x_i p_i=p_i^2x_ip_i^2$ whenever $x_i \in (A_i)_+$, for all
    $i\in I$. That is, $(p_i) \in BP(\el p(A_i))$ if and only if $p_i \in BP(A_i)$.

    Similarly, one can check that
    \[
    BP_l(\el p(A_i))=\prod_{i \in I} BP_l(A_i)\quad\text{and}\quad
    BP_r(\el p(A_i))=\prod_{i \in I} BP_r(A_i).
    \]
    If each $A_i$ has an identity, then also
    \[
    OI(\el \infty (A_i))=\prod_{i \in I} OI(A_i).
    \]
\end{rem}

\begin{example}\label{ex:oinotbp2}
    Let $A$ be the Banach lattice algebra of \cref{ex:oinotbp}, and
    consider $\el \infty (A,A)$.
    By previous remark, the element
    \[
        q=\left( \begin{pmatrix} 1&0\\0&0 \end{pmatrix} ,\begin{pmatrix}
    0&1\\0&0\end{pmatrix}  \right)
    \]
    is a band projection that is not order idempotent. In this case,
    $q^2\neq 0$ and $L_qR_q\neq 0$.
\end{example}

At this point, it is clear that the relation between order idempotents
and band projections in Banach lattice algebras with identity is not
trivial. The forthcoming results aim at a better understanding of it.

\begin{prop}
    Let $A$ be a Banach lattice algebra with identity $e$, and let $p \in
    BP(A)$. Then the following are equivalent:
    \begin{enumerate}
        \item $p \in OI(A)$,
        \item $p$ is idempotent,
        \item $p \in A_e$,
        \item $(\lambda e+p)^{-1}\ge 0$ for some $\lambda >\|p\|$.
    \end{enumerate}
\end{prop}
\begin{proof}
    Condition $(i)$ clearly implies $(ii)$. If $p$ is idempotent,
    $L_pR_p(e)=p^2=p$, and since $0\le L_pR_p(e)\le e$, we have $p \in
    A_e$. So $(ii)$ implies $(iii)$. If $p \in A_e$, then, using that
    $A_e$ is a subalgebra and sublattice of $A$, we can consider the
    restriction $L_pR_p\colon A_e\to A_e$. This operator
    is a band projection on $A_e$. Represent $A_e$ as a
    $C(K)$, for a certain compact Hausdorff space $K$. Since
    $BP(C(K))=OI(C(K))$, it follows that $p \in OI(A_e)$. But the
    order and product in $A_e$ are the same as those in $A$, thus $p \in OI(A)$.
    This shows that $(iii)$ implies $(i)$.

    It only remains to check the equivalence between $(iii)$ and $(iv)$.
    Since $p \in BP(A)$, we have $p^2=L_pR_p(e)=(L_pR_p)^2(e)=p^{4}$. Inductively, $p^2=p^{2n}$ and $p^3=p^{2n+1}$ for all
    $n\ge 1$. Using this fact and the Neumann series, for every
    $\lambda >\max\{1,\|p\|\}$,
    \begin{align*}
        \lambda (\lambda e+p)^{-1}=(e+p/\lambda
        )^{-1}&=\sum_{n=0}^{\infty}\frac{(-p)^{n}}{\lambda ^{n}}\\
              &=e-\frac{p}{\lambda }+\frac{p^2}{\lambda
              ^2}\sum_{n=0}^{\infty}\frac{1}{\lambda
              ^{2n}}-\frac{p^3}{\lambda
          ^3}\sum_{n=0}^{\infty}\frac{1}{\lambda ^{2n}}\\
              &=e-\frac{p}{\lambda }+\frac{p^2}{\lambda
              ^2-1}-\frac{p^3}{\lambda (\lambda ^2-1)}\\
              &=\left( e-\frac{p}{\lambda} \right) \left(
              e+\frac{p^2}{\lambda ^2-1} \right)
    \end{align*}
    Suppose that $p \in A_e$. Then $p\le \|p\|_e e$ and, by the
    previous computation, $(\|p\|_e+2)e+p$ is invertible with positive
    inverse. This shows that $(iii)$ implies $(iv)$. Conversely, if
    ${(\lambda e+p)^{-1}}\ge 0$ for some $\lambda >\|p\|$, either $p=0$,
    in which case the result is trivial, or
    $p\neq 0$, in which case $\lambda >\|p\|_e\ge 1$\footnote{Due to the fact that the spectral
    radius of $p$ is $1$, see \cref{sec:bla_spectra}.} and then
    \[
        (\lambda e+p)(\lambda e-p)=\lambda ^2e-p^2\ge 0.
    \]
    Multiplying by $(\lambda e+p)^{-1}$ on the left
    \[
    \lambda e-p=(\lambda e+p)^{-1}(\lambda ^2e-p^2)\ge 0
    \]
    because the product of positive elements is positive. Hence $p \in
    A_e$.
\end{proof}

In order to delve deeper into the relation between band projections and order
idempotents, we need to introduce spectra in our study.

\subsection{Complexification and spectra}\label{sec:bla_spectra}

To talk about spectra and see their relation with band
projections and order idempotents, we need to work over the complex
field. So first we need to ``complexify'' some of the notions introduced
so far.

Let $A$ be a Banach lattice algebra. Applying the procedure for the
complexification of Banach lattices, as exposed in
\cref{sec:regop_spectra}, one
obtains the complex Banach lattice $A_\C$ with norm $\|z\|_\C=\| |z|
\|$. With the natural product
\[
    (x_1+ix_2)(y_1+iy_2)=(x_1y_1-x_2y_2)+i(x_1y_2+x_2y_1)\quad\text{where
    }x_1,x_2,y_1,y_2 \in A,
\]
the space $A_\C$ becomes a complex algebra, that has an identity precisely
when $A$ does. This product is compatible
with the complex lattice structure in the sense that $|z_1 z_2|\le
|z_1| |z_2|$, for all $z_1,z_2 \in A_\C$. This fact is not trivial,
and a proof may be found in \cite{huijsmans1985}. The submultiplicativity of the
norm on $A$ immediately implies $\|z_1 z_2\|_\C\le \|z_1\|_\C
\|z_2\|_\C$ for $z_1,z_2 \in A_\C$. Hence $A_\C$ is also a complex
Banach algebra with norm $\|{\cdot }\|_\C$. The space $A_\C$ with its
structures of complex Banach lattice and complex Banach algebra is
called a \emph{complex Banach lattice algebra}. When $A$ has an identity, we
emphasize it by saying that $A_\C$ is a \emph{complex Banach
lattice algebra with identity}.

\begin{example}
    \begin{enumerate}
        \item If $K$ is a compact Hausdorff space, the
            complexification of $C(K)$ may be identified with
            $C(K,\C)$, the space of continuous complex-valued
            functions on~$K$.
        \item If $X$ is an order complete Banach lattice, the
            complexification of $\Lr X$ may be identified with the
            space of regular operators over $X_\C$ (see
            \cref{sec:regop_spectra}).
    \end{enumerate}
\end{example}

Next we want to present a result analogous to \cref{thm:bla_ideal} in
complex Banach lattice algebras. If $X$ is a Banach lattice, a
subspace $I\subseteq X_\C$ is called an \emph{ideal} if $x \in I$ and
$|y|\le |x|$ imply $y \in I$.
It is not difficult to check that $I$ is the complexification of the
ideal $I\cap X$ of $X$, and that the intersection of all the ideals
containing $z \in X_\C$ is given by
\[
    (I_z)_{\C}=\{\, w \in X_\C : |w|\le \lambda |z|\text{ for some }\lambda \in
\R_+ \, \}.
\]
This is called the \emph{ideal generated by $z$}.

An ideal $B\subseteq X_\C$ is called a
\emph{band} (resp.\ \emph{projection band}) if $B\cap X$ is a band
(resp.\ projection band) in $X$. Again, one can check that an ideal
$B\subseteq X$ is a projection band if and only if $X_\C=B\oplus
B^{d}$, where
\[
B^{d}=\{\, z \in X_\C : |w|\wedge |z|=0 \text{ for all }w \in B \, \}.
\]

Finally, some common notation: for $z \in A_\C$ denote its spectrum by
$\sigma (z)$ and its spectral radius by $r(z)$.

\begin{thm}\label{thm:bla_ideal_complex}
    Let $A_\C$ be a complex Banach lattice algebra with identity $e$ and let
    $(A_\C)_e$ be the ideal generated by $e$ in $A_\C$. Then
    $(A_\C)_e=(A_e)_\C$, i.e., $(A_\C)_e$ is the complexification of
    the Banach lattice algebra $A_e$. Moreover:
    \begin{enumerate}
        \item The space $(A_\C)_e$ can
            be isometrically identified with $C(K,\C)$, for some compact Hausdorff $K$.
        \item The ideal $(A_\C)_e$ is a projection band in $A_\C$.
        \item The space $(A_\C)_e$ is an inverse closed subalgebra of $A_\C$.
    \end{enumerate}
\end{thm}
\begin{proof}
    Since $(A_\C)_e\cap A=A_e$, it follows that $(A_\C)_e=(A_e)_\C$.
    It is then a direct consequence of \cref{thm:bla_ideal} that $(A_\C)_e$, with the
    complexification of the norm $\|{\cdot }\|_e$, is isometric to
    $C(K,\C)$. To prove $(i)$ it only remains to check that the norm
    $\|{\cdot }\|_{e,\C}$ is equal to $\|{\cdot }\|_\C$.
    Being $A_e$ an inverse closed
    subalgebra, it is easy to check that the spectrum of an element in
    $A_e$ (i.e., its range when viewed as a function in $C(K,\C)$) is
    the same as its spectrum in $A_\C$. Therefore $\|z\|\ge r(z)=\|z\|_e$,
    and since we are assuming $\|e\|=1$ in our definition of Banach
    lattice algebras, $\|z\|\le \|z\|_e$. Joining both inequalities,
    $\|z\|=\|z\|_e=r(z)$ and $(i)$ is proved.
    From \cref{thm:bla_ideal} clearly follows $(ii)$, and exactly the
    same proof as in \cref{thm:bla_ideal} can be used to show $(iii)$.
\end{proof}

Previous result completes the generalization from
\cref{thm:center} to \cref{thm:bla_ideal}. It also completes
\cref{ex:ALalg}. Some remarks regarding the spectrum of order
idempotents and band projections are in order.

\begin{rem}
    \begin{enumerate}
        \item Let $p$ be an order idempotent. When we represent
            $(A_\C)_e$ as a space of continuous functions, $p$ is the
            characteristic function of a clopen set. As such,
            $\sigma (p)=\{0,1\}$ whenever $p\neq 0,e$. By previous
            theorem, $\|p\|=r(p)=1$.
        \item If $p \in BP(A)$, either
            $p^2=0$, in which case $r(p)=0$ and $\sigma (p)=\{0\}$ (the complexification of
            \cref{ex:oinotbp} serves as an example of this), or $p^2\neq 0$, in which case
            $\|p\|=r(p)=\lim_n \|p^{2n}\|^{1/2n}=1$ (because $p^2$ is order
            idempotent, so $\|p^2\|=1$, and also $(p^2)^{n}=p^2$).
            In both cases, $\sigma
            (p^2)\subseteq \{0,1\}$ implies $\sigma (p)\subseteq \{-1,0,1\}$.
        \item For $p \in BP(A)$ to be an order idempotent, it is
            necessary that $\sigma (p)\subseteq \{0,1\}$. However, it is not
            sufficient: one can easily check that the element $q$ in
            \cref{ex:oinotbp2} satisfies $\sigma (q)=\{0,1\}$, yet it
            is not an order idempotent.
    \end{enumerate}
\end{rem}

The following is an example of a band projection with spectrum
$\{0,1,-1\}$.

\begin{example}
    In the Banach lattice algebra $M_3(\R)$ (see \cref{ex:oinotbp}), consider the 3-dimensional subspace:
    \begin{align*}
        A&=\barespn \left\{ \begin{pmatrix} 1&0&0\\0&0&0\\0&0&1
            \end{pmatrix} ,\begin{pmatrix} 0&0&1\\0&0&0\\1&0&0
        \end{pmatrix}, \begin{pmatrix} 0&0&0\\0&1&0\\0&0&0
\end{pmatrix}  \right\} \\
         &=\left\{ \begin{pmatrix} x&0&y\\0&z&0\\y&0&x \end{pmatrix}
         \,:\,x,y,z \in \R \right\}.
    \end{align*}
    Being generated by pairwise disjoint elements, $A$ is a
    sublattice. Moreover, the identities
    \[
    \begin{pmatrix} x&0&y\\0&z&0\\y&0&x \end{pmatrix}
    \begin{pmatrix} x'&0&y'\\0&z'&0\\y'&0&x' \end{pmatrix}=
    \begin{pmatrix} x x'+y y'&0&xy'+x'y\\0&z z'&0\\xy'+x'y&0&x x'+y y'
    \end{pmatrix},
    \]
    \[
    \begin{pmatrix} x&0&y\\0&z&0\\y&0&x \end{pmatrix}^{-1}=
    \frac{1}{x^2-y^2}\begin{pmatrix}
    x&0&-y\\0&z^{-1}(x^2-y^2)&0\\-y&0&x \end{pmatrix},\quad \text{when
    both }x\neq y\text{ and }
    z\neq 0,
    \]
    prove that $A$ is an inverse closed subalgebra, with identity because the
    identity matrix is contained in $A$.
    Therefore, $A$ is a Banach lattice algebra with the inherited product, order and
    norm. Its complexification consists in taking matrices with
    complex entries.

    Set
    \[
    p=\begin{pmatrix} 0&0&1\\0&0&0\\1&0&0 \end{pmatrix}.
    \]
    It is not difficult to check that $pap=a$ for every $a \in
    A_\C$.\footnote{Using elementary matrix theory: right multiplication by $p$
    swaps the first and third columns of $a \in A_\C$, while left
    multiplication swaps the first and third rows of $a$; that is,
    $pap$ is the reflection of $a$ with respect to both of its
    diagonals. The elements of $A$, however, are invariant under this
    operation.} Therefore $L_pR_p=I_{A_\C}$, and this implies $p \in
    BP(A)$. The spectrum of $p$ is the set of roots of its
    characteristic polynomial
    \[
    \det(p-\lambda I_3)=\begin{vmatrix} -\lambda &0&1\\0&-\lambda &0\\1&0&-\lambda
    \end{vmatrix}=-\lambda (\lambda +1)(\lambda -1),
    \]
    which is $\sigma_A (p)=\sigma (p)=\{-1,0,1\}$.
\end{example}

Next we want to provide a result that relates spectra with order
idempotents. We start with a simple observation.

\begin{rem}
    In $C(K)$ order idempotents are characterized by their spectrum.
    More precisely, $p \in OI(C(K))$ if and only if $\sigma
    (p)\subseteq \{0,1\}$.
\end{rem}

We also need the generalization of \cref{thm:swa} to arbitrary
complex Banach lattice algebras with identity.

\begin{cor}
    Let $A_\C$ be a complex Banach lattice algebra with identity. Let $a \in A_+$ be an
    invertible element with positive inverse. Then
    $a \in A_e$ if and only if $\sigma (a)\subseteq \R_+$.
\end{cor}
\begin{proof}
    Let $a \in A_+$ be an invertible element such that $a ^{-1}\in
    A_+$. Then $R_a\colon A\to A$, being a positive operator with
    positive inverse $R_{a^{-1}}$, is a lattice isomorphism. Note that
    \[
        (R_a-\lambda I_{A_\C})(x)=x(a-\lambda e)=R_{a-\lambda
        e}(x)\quad\text{for all }x \in A_\C,
    \]
    where $e$ denotes the identity element of $A_\C$. It follows that
    $\sigma (R_a)\subseteq \sigma (a)$, since $R_{a-\lambda e}$ is
    invertible whenever $a-\lambda e$ is.
    Suppose $\sigma (a)\subseteq \R_+$. Then $\sigma (R_a)\subseteq \R_+$ and
    $R_a$ is a lattice isomorphism. By \cref{thm:swa}, $R_a \in
    \mathcal{Z}(A)$, which readily implies $a \in A_e$.
    The converse is direct if we see $a \in A_e$ as a positive
    continuous function on a compact space.
\end{proof}

\begin{prop}
    Let $A_\C$ be a complex Banach lattice algebra with identity. Let $a
    \in A_+$ be an invertible element such that:
    \begin{enumerate}
        \item $a ^{-1}\in A_+$,
        \item $\sigma (a)\subseteq \{\lambda ,\lambda +1\}$ for some
            $\lambda \in \R_+$.
    \end{enumerate}
    Then $a-\lambda e \in OI(A)$, where $e$ is the identity of $A$.
\end{prop}
\begin{proof}
    Since $a$ has positive spectrum, and $a^{-1} \in A_+$, from
    previous corollary it follows
    that $a \in A_e$. Then $a-\lambda e \in A_e$, and its
    spectrum satisfies, by hypothesis, $\sigma (a-\lambda e)\subseteq
    \{0,1\}$. Since $A_e$ is essentially a space of continuous
    functions, it follows that $a-\lambda e$ is an order idempotent.
\end{proof}

\subsection{Inner band projections}\label{sec:bla_innerbp}

In this section the construction of inner band projections
explained
in \cref{sec:regop_innerbp} is extended to general Banach
lattice algebras. Let $A$ be an
order complete Banach lattice algebra and let $\{p_\lambda \}_{\lambda \in \Lambda
}\subseteq BP_l(A)\cap BP_r(A)$ be such that $p_\alpha p_\beta
=\delta _{\alpha \beta }p_\alpha $ for all $\alpha ,\beta \in \Lambda
$ (in particular, if $A$ has an identity,
we can take the $p_\lambda $ to be pairwise disjoint order
idempotents).
For $\Gamma \subseteq \Lambda \times \Lambda $ define the map
\begin{equation}\label{eq:innerbp}
\begin{array}{cccc}
P_\Gamma \colon& A_+ & \longrightarrow & A_+ \\
        & x & \longmapsto & \displaystyle\bigvee_{(\alpha ,\beta )\in \Gamma }
        p_\alpha xp_\beta  \\
\end{array}.
\end{equation}
Note that this is well-defined, because $p_\alpha xp_\beta
=L_{p_\alpha }R_{p_\beta }(x)\le x$ (recall that $L_{p_\alpha }R_{p_\beta }$
is a band projection on $A$), and $A$ is order complete. If we
do not assume that $A$ is order complete, the coming results work just
fine, as long as we ask for this supremum to exist. In particular, if
$\Gamma $ is finite, everything works out.

\begin{lem}
    Let $A$ be a Banach lattice algebra, and let $(p_\lambda
    )_{\lambda \in \Lambda }\subseteq BP_l(A)\cap BP_r(A)$ be such
    that $p_\alpha p_\beta =\delta _{\alpha \beta }p_\alpha $. If
    $\Phi \subseteq \Lambda \times \Lambda $ is finite, then
    \[
    \bigvee_{(\alpha ,\beta )\in \Phi } p_\alpha xp_\beta
    =\sum_{(\alpha ,\beta )\in \Phi } p_\alpha xp_\beta
    ,\quad\text{for all }x \in A_+.
    \]
\end{lem}
\begin{proof}
    For every $(\alpha ,\beta )\in \Phi $, $L_{p_\alpha }R_{p_\beta }$
    is a band projection; to ease the notation, denote it by
    $L_\alpha R_\beta $. Whenever $(\gamma ,\delta )\neq (\alpha
    ,\beta )$:
    \[
        (L_\alpha R_\beta )(L_\gamma R_\delta )(x)=p_\alpha p_\gamma x
        p_\delta p_\beta =0
    \]
    because either $p_\alpha p_\gamma =0$ or $p_\delta p_\beta =0$.
    Therefore $(L_\alpha R_\beta )(L_\gamma R_\delta )=0$. Being band
    projections, this implies that the bands determined by $L_\alpha
    R_\beta $ and $L_\gamma R_\delta $ are disjoint. Thus
    \[
    p_\alpha xp_\beta \wedge p_\gamma xp_\delta =(L_\alpha R_\beta
    )(x)\wedge (L_\gamma R_\delta )(x)=0,
    \]
    or, equivalently, $p_\alpha xp_\beta \vee p_\gamma xp_\delta
    =p_\alpha xp_\beta +p_\gamma xp_\delta $. Since this is true for
    any pair of indices, the result follows by induction.
\end{proof}

\begin{thm}\label{thm:innerbp}
    The map $P_\Gamma $, as defined in \eqref{eq:innerbp}, extends to
    a unique band projection on $A$.
\end{thm}
\begin{proof}
    Note first that, using previous lemma,
    \begin{equation}\label{eq:ordlim}
    P_\Gamma x=\bigvee_{(\alpha ,\beta )\in \Gamma } p_\alpha xp_\beta
    =\sup_{\Phi \in \Fin(\Gamma )}
    \bigvee_{(\alpha ,\beta )\in \Phi }^{}p_\alpha xp_\beta
    =\sup_{\Phi \in \Fin(\Gamma )}
    \sum_{(\alpha ,\beta )\in \Phi }^{}p_\alpha xp_\beta .
    \end{equation}
    It follows from the definition that $P_\Gamma (x+y)\le P_\Gamma
    (x)+P_\Gamma (y)$. For every $\Phi,\Psi \in \Fin(\Gamma )$:
    \begin{align*}
        \sum_{(\alpha ,\beta )\in \Phi }^{}p_\alpha xp_\beta+
        \sum_{(\gamma ,\delta )\in \Psi  }^{}p_\gamma  yp_\delta&\le
        \sum_{(\alpha ,\beta )\in \Phi \cup \Psi }^{}p_\alpha xp_\beta
        +\sum_{(\gamma ,\delta )\in \Phi \cup \Psi }^{}p_\gamma
        yp_\delta \\ &=\sum_{(\alpha ,\beta)\in \Phi \cup \Psi
    }^{}p_\alpha (x+y)p_\beta \le P_\Gamma (x+y).
    \end{align*}
    Taking suprema separately in $\Phi $ and $\Psi $ we conclude that
    $P_\Gamma $ is additive. Hence it extends to a unique positive
    operator $P_\Gamma \colon A\to A$. It is clear from the definition that $0\le P_\Gamma \le
    I_A$. To show that it is a band projection, it only remains to
    prove that it is idempotent. By previous inequality,
    $P_\Gamma ^2\le P_\Gamma $. To see the other inequality, note that
    for an arbitrary $\Phi \in \Fin(\Gamma )$:
    \begin{align*}
        P_\Gamma (P_\Gamma (x))\ge \sum_{(\alpha ,\beta )\in \Phi
        }^{}p_\alpha P_\Gamma (x)p_\beta &\ge \sum_{(\alpha ,\beta
        )\in \Phi }^{} p_\alpha \bigg( \sum_{(\gamma  ,\delta  )\in \Phi 
        }^{}p_\gamma xp_\delta  \bigg) p_\beta \\
        &=\sum_{(\alpha ,\beta )\in \Phi  }^{}\sum_{(\gamma ,\delta
        )\in \Phi  }^{}p_\alpha p_\gamma xp_\delta p_\beta \\
        &=\sum_{(\alpha ,\beta )\in \Phi }^{}p_\alpha
        xp_\beta .
    \end{align*}
    Inequality $P_\Gamma ^2\ge P_\Gamma $ follows from taking supremum
    on $\Phi \in \Fin(\Gamma )$.
\end{proof}

\begin{defn}
    Let $A$ be a Banach lattice algebra and let $P\colon A\to A$ be a
    band projection. If there exist $(p_\lambda )_{\lambda \in \Lambda
    }\subseteq BP_l(A)\cap BP_r(A)$ such that $p_\alpha p_\beta
    =\delta _{\alpha \beta }p_\alpha $ for all $\alpha ,\beta \in
    \Lambda $, and such that $P=P_\Gamma $ for some $\Gamma \subseteq
    \Lambda \times \Lambda $ as defined in \cref{eq:innerbp}, then we
    say that $P$ is an \emph{inner band projection}. We say that the
    band $P(A)$ is an \emph{inner projection band}.
\end{defn}

\begin{example}
    Let $K$ be a compact Hausdorff space. Even though $C(K)$ may not
    be order complete, it still makes sense to consider inner band
    projections, as long as the supremum in \cref{eq:innerbp} is
    well-defined. In fact, in $C(K)$, any band projection is inner,
    because it consists in multiplying by the characteristic function
    of a clopen set, which is an order idempotent.
\end{example}

\begin{rem}
    Let $A$ be an order complete Banach lattice algebra.
    Since in $\Lr A$ the supremum of an increasing net is taken pointwise,
    equation \eqref{eq:ordlim} implies
    \[
    P_\Gamma =\olim_{\Phi \in \Fin(\Gamma )}\sum_{(\alpha ,\beta )\in
    \Phi }^{}L_\alpha R_\beta,
    \]
    where $\olim$ denotes the limit of the net in the order convergence (see
    \cite[Section 2.4]{troitsky}).
    In particular, if $A$ is order continuous, we have
    \[
    P_\Gamma (x)=\lim_{\Phi \in \Fin(\Gamma )}\sum_{(\alpha ,\beta
    )\in \Phi }^{}p_\alpha xp_\beta =\sum_{(\alpha ,\beta )\in \Gamma
    }^{}p_\alpha xp_\beta .
    \]
\end{rem}

\begin{prop}
    Let $A$ be an order complete Banach lattice algebra, and let $(p_\lambda
    )_{\lambda \in \Lambda }\subseteq BP_l(A)\cap BP_r(A)$ be such
    that $p_\alpha p_\beta =\delta _{\alpha \beta }p_\alpha $.
    The set $\{\, P_\Gamma  : \Gamma \subseteq \Lambda \times \Lambda
    \, \} $ forms a Boolean algebra with maximum element $P_{\Lambda
    \times \Lambda }$, minimum $P_\O$ and operations
    \[
    P_\Gamma \wedge P_\Delta =P_{\Gamma \cap \Delta },\;
    P_\Gamma \vee P_\Delta =P_{\Gamma \cup  \Delta },\;
    \overline{P}_\Gamma =P_{\overline{\Gamma }},\quad \text{for
    }\Gamma ,\Delta \subseteq \Lambda \times \Lambda ,
    \]
    where $\overline{\Gamma }=\Lambda \times \Lambda \setminus \Gamma
    $ denotes the complementary set.
\end{prop}
\begin{proof}
    On the set of inner band projections, consider the usual order of
    operators. Let $\Gamma ,\Delta \subseteq \Lambda \times \Lambda $.
    We first want to prove that $P_{\Gamma }\wedge P_\Delta =P_{\Gamma
    \cap \Delta }$. For every $x \in A_+$, the net $(\sum_{(\gamma
    ,\delta )\in \Psi }^{} p_\gamma xp_\delta )_{\Psi \in \Fin(\Delta
    )}$ increases to $P_\Delta (x)$.
    For every $(\alpha ,\beta )\in \Gamma $, using that band
    projections are order continuous:
    \[
    L_{p_\alpha }R_{p_\beta }(P_\Delta (x))=\olim_{\Psi \in
    \Fin(\Delta )} \sum_{(\gamma ,\delta )\in \Psi }^{}L_{p_\alpha
    }R_{p_\beta }(p_\gamma xp_\delta )=
    \begin{cases}
        p_\alpha xp_\beta &\text{if }(\alpha ,\beta )\in \Delta ,\\
        0&\text{if }(\alpha ,\beta )\not\in \Delta .
    \end{cases}
    \]
    Therefore
    \[
    P_\Gamma (P_\Delta (x))=\bigvee_{(\alpha ,\beta )\in \Gamma }
    L_{p_\alpha }R_{p_\beta }(P_\Delta (x))=\bigvee_{(\alpha ,\beta
    )\in \Gamma \cap \Delta } p_\alpha xp_\beta =P_{\Gamma \cap \Delta
    }(x),
    \]
    as wanted

    To check the formula for the supremum, first consider the
    situation $\Gamma \cap \Delta =\O$. In this case, $P_\Gamma
    P_\Delta =0$, and
    \[
    P_\Gamma +P_\Delta =P_\Gamma \vee P_\Delta \le P_{\Gamma \cup \Delta }
    \]
    because both $P_\Gamma \le P_{\Gamma \cup \Delta }$ and $P_\Delta
    \le P_{\Gamma \cup \Delta }$. To see the reverse inequality, take
    finite sets $\Phi \in \Fin(\Gamma )$ and $\Psi \in \Fin(\Delta )$.
    Using that they are disjoint:
    \[
    P_\Gamma (x)+P_\Delta (x)\ge \sum_{(\alpha ,\beta )\in \Phi }^{}
    p_\alpha xp_\beta +\sum_{(\gamma ,\delta )\in \Psi }^{}p_\gamma
    xp_\delta =\sum_{(\alpha ,\beta )\in \Phi \cup \Psi }^{}p_\alpha
    xp_\beta.
    \]
    The right hand side is an arbitrary finite subset of $\Gamma \cup
    \Delta $; taking supremum over these yields $P_\Gamma
    +P_\Delta \ge P_{\Gamma \cup \Delta }$. In particular, when
    $\Delta
    =\overline{\Gamma }$, we get the desired formula for the complementary.
    Finally, for arbitrary $\Gamma $ and $\Delta $ (not necessarily
    disjoint), write their union as the three disjoint unions $\Gamma \cup \Delta =(\overline{\Gamma }\cap
    \Delta )\cup (\Gamma \cap \overline{\Delta })\cup (\Gamma \cap
    \Delta )$ and apply the previous formulae.
\end{proof}

\begin{rem}
    When $A=\Lr X$, for some order complete Banach lattice $X$, the
    Boolean algebra of inner band projections is isomorphic to
    $2^{\Lambda \times \Lambda }$, according to \cref{prop:innerbp_boole}. In the general case, however, this
    need not be true. In \cref{ex:pathologicalc}, the
    only inner band projection is the zero projection, even though there are many pairwise
    disjoint left and right band projections.
\end{rem}

We finish with an illustration of inner band projections in a
Banach lattice algebra without identity. This example also shows that an
analogous of \cref{thm:innerbp_all} for this general setup is not
straightforward.

\begin{example}
    Consider the 3-dimensional Banach lattice algebra
    \begin{align*}
        A&=\barespn \left\{ \begin{pmatrix} 1&0&0\\0&0&0\\0&0&0
                \end{pmatrix},\begin{pmatrix} 0&0&0\\0&1&0\\0&0&0
                \end{pmatrix},\begin{pmatrix} 0&0&0\\0&0&0\\1&0&0
                \end{pmatrix}  \right\}\\ &=\left\{\, \begin{pmatrix} x&0&0\\0&y&0\\z&0&0 \end{pmatrix}  :
        x,y,z \in \R \, \right\} \subseteq M_3(\R),
    \end{align*}
    where the order, product and norm are inherited from $M_3(\R)$
    (see \cref{ex:oinotbp}). It
    is not difficult to check that this Banach lattice algebra does
    not have an identity.
    By direct computation, one can also check that
    \[
        BP_l(A)\cap BP_r(A)=\left\{\,  \begin{pmatrix}
        0&0&0\\0&0&0\\0&0&0 \end{pmatrix} ,p_1=\begin{pmatrix}
        1&0&0\\0&0&0\\0&0&0 \end{pmatrix},p_2=\begin{pmatrix}
        0&0&0\\0&1&0\\0&0&0 \end{pmatrix},\begin{pmatrix}
        1&0&0\\0&1&0\\0&0&0 \end{pmatrix}  \, \right\} .
    \]
    Since $p_1ap_2=0=p_2ap_1$ for $a \in A$, the only inner band projections besides
    the zero projection are $a\mapsto p_1ap_1$,
    $a\mapsto p_2ap_2$, and their sum; explicitly, the only inner band
    projections are
    \[
        \begin{pmatrix} x&0&0\\0&y&0\\z&0&0 \end{pmatrix} \mapsto
        \begin{pmatrix} x&0&0\\0&0&0\\0&0&0 \end{pmatrix} ,\quad
        \begin{pmatrix} x&0&0\\0&y&0\\z&0&0 \end{pmatrix} \mapsto
        \begin{pmatrix} 0&0&0\\0&y&0\\0&0&0 \end{pmatrix} ,\quad
        \begin{pmatrix} x&0&0\\0&y&0\\z&0&0 \end{pmatrix} \mapsto
        \begin{pmatrix} x&0&0\\0&y&0\\0&0&0 \end{pmatrix} .
    \]
    It is easy to check that
    \[
        \begin{pmatrix} x&0&0\\0&y&0\\z&0&0 \end{pmatrix} \mapsto
        \begin{pmatrix} 0&0&0\\0&0&0\\z&0&0 \end{pmatrix}
    \]
    defines a band projection that is not inner.
\end{example}

\section*{Acknowledgements}
The author is very thankful to Pedro Tradacete for valuable
discussions, ideas and corrections.
Research partially supported by an FPI--UAM 2023 contract funded by
Universidad Autónoma de Madrid.

\printbibliography

\end{document}